\documentclass[sts]{imsart}

\usepackage{latexsym,amssymb}
\usepackage{amsmath}
\usepackage[mathscr]{eucal}
\usepackage{epic,eepic}

\usepackage{color}
\usepackage{enumerate}
\usepackage{comment}
\usepackage[
 anchorcolor=blue,%
 bookmarks=true,%
 bookmarksnumbered=true,%
 colorlinks=true,%
 citecolor=cyan,%
 linkcolor=blue%
]{hyperref}

\definecolor{refkey}{gray}{0.5}
\definecolor{labelkey}{gray}{0.2}

\usepackage{amsthm}
\newtheorem{theorem}{Theorem}[section]
\newtheorem{proposition}[theorem]{Proposition}
\newtheorem{lemma}[theorem]{Lemma}
\newtheorem{corollary}[theorem]{Corollary}

\theoremstyle{definition}

\theoremstyle{remark}
\newtheorem{remark}[theorem]{Remark}

\makeatletter

\@addtoreset{equation}{section}
\makeatother

\newcommand{\N}{\mathbb{N}}
\newcommand{\R}{\mathbb{R}}
\newcommand{\Sph}{\mathbb{S}}
\newcommand{\bb}{\mathbf{b}}
\newcommand{\fm}{\mathfrak{m}}
\newcommand{\fn}{\mathfrak{n}}
\newcommand{\fq}{\mathfrak{q}}
\newcommand{\dist}{\mathrm{d}}

\newcommand{\ve}{\varepsilon}

\newcommand{\lra}{\longrightarrow}
\newcommand{\e}{\mathrm{e}}

\newcommand{\CD}{\mathrm{CD}}
\newcommand{\RCD}{\mathrm{RCD}}
\newcommand{\Ric}{\mathrm{Ric}}

\newcommand{\supp}{\mathrm{supp}}

\newcommand*\diff{\mathop{}\!\mathrm{d}}

\usepackage{cite}
\usepackage[utf8]{inputenc}
\usepackage{dsfont}
\definecolor{purple}{rgb}{0.62, 0.0, 0.77}

\newcommand{\veb}[1]{\textcolor{blue}{\textbf{VEB:}  #1}}

\newcommand{\js}[1]{\textcolor{purple}{\textbf{JS:}  #1}}

\begin{document}

\begin{frontmatter}

\title{A generalization of Gr\"{u}nbaum's inequality in RCD$(0,N)$-spaces}

\author{Victor-Emmanuel Brunel\thanks{CREST ENSAE IP Paris, 5 Av. Le Chetelier 91120 Palaiseau, France ({\sf victor.emmanuel.brunel@ensae.fr})}
\and
Shin-ichi Ohta\thanks{
Department of Mathematics, Osaka University,
Osaka 560-0043, Japan \&
RIKEN Center for Advanced Intelligence Project (AIP),
1-4-1 Nihonbashi, Tokyo 103-0027, Japan
({\sf s.ohta@math.sci.osaka-u.ac.jp})}
\and
Jordan Serres\thanks{CREST ENSAE IP Paris, 5 Av. Le Chetelier 91120 Palaiseau, France
({\sf jordan.serres@ensae.fr})}}

\setattribute{abstractname}{skip} {{\bf Abstract.} } 
\begin{abstract}
We generalize Gr\"{u}nbaum's classical inequality in convex geometry to curved spaces with nonnegative Ricci curvature,
precisely, to $\RCD(0,N)$-spaces with $N \in (1,\infty)$ as well as weighted Riemannian manifolds of $\Ric_N \ge 0$ for $N \in (-\infty,-1) \cup \{\infty\}$.
Our formulation makes use of the isometric splitting theorem; given a convex set $\Omega$ and the Busemann function associated with any straight line, the volume of the intersection of $\Omega$ and any sublevel set of the Busemann function that contains a barycenter of $\Omega$ is bounded from below in terms of $N$. We also extend this inequality beyond uniform distributions on convex sets. Moreover, we establish some rigidity results by using the localization method, and the stability problem is also studied.
%
\end{abstract}
\end{frontmatter}


\section{Introduction}

\subsection{Background}

In the Euclidean space $\R^n$, it is a well known fact that, given any probability measure, there exists a point such that any closed halfspace including this point has mass at least $1/(n+1)$ \cite[Theorem 1]{Gr}, \cite[Lemma 6.3]{donoho1992breakdown}.
This bound is tight since there are probability measures (e.g., the uniform distribution on the vertex set of a simplex) such that every point is in a closed halfspace with mass at most $1/(n+1)$.
Hence, this result becomes very little informative when the dimension $n$ is very large because it does not allow to discriminate between points in general.
Nonetheless, the bound can be improved, provided that the measure satisfies some geometric properties.
For instance, for the uniform distribution on a bounded convex set $\Omega \subset \R^n$, Gr\"unbaum's inequality \cite[Theorem 2]{Gr} states that any closed halfspace including the barycenter (centroid) of $\Omega$ must have volume at least $(n/(n+1))^n$.
Gr\"unbaum's inequality remains informative in any dimension since $(n/(n+1))^n \ge \e^{-1}$,
and it has been extended to log-concave distributions (among which uniform distributions on convex sets) \cite[Lemma 5.12]{lovasz2007geometry}.
Namely, for any log-concave probability measure on $\R^n$, any closed halfspace containing the barycenter must have mass at least $\e^{-1}$, independently of the dimension $n$.
By virtue of the Pr\'ekopa--Leindler inequality, it is in fact sufficient to consider the one-dimensional case in order to prove the inequality for log-concave measures \cite{lovasz2007geometry}.
Furthermore, as we will see in this work, a family of functional inequalities that generalize the Pr\'ekopa--Leindler inequality, namely the Borell--Brascamp--Lieb inequalities, allow to generalize the bound even further, to broader classes of distributions. 

The problem described here is directly related to a notion that is central in descriptive statistics, called Tukey's depth.
Given a probability distribution $\mu$ on $\R^n$, Tukey's depth of a point $x\in \R^n$ relative to $\mu$ is defined as $D_\mu(x):=\inf_{H}\mu(H)$, where the infimum is taken over all closed halfspaces $H \subset \R^n$ containing $x$.
Then, the inequalities discussed above can be summarized as follows: 
\begin{itemize}
    \item For any distribution $\mu$ on $\R^n$, there exists a point $x\in\R^n$ with $D_\mu(x)\geq 1/(n+1)$; moreover, there exists $\mu$ for which $\sup_{x\in\R^n}D_\mu(x)=1/(n+1)$.
    \item If $\mu$ is the uniform distribution on a convex set $\Omega$ in $\R^n$, then there exists $x\in\R^n$ with $D_\mu(x)\geq (n/(n+1))^n$ and $x$ can be chosen to be the barycenter of $\Omega$.
    \item If $\mu$ is a log-concave distribution on $\R^n$, then there exists $x\in\R^n$ with $D_\mu(x)\geq \e^{-1}$ and one can choose $x$ to be the barycenter of $\mu$. 
\end{itemize}
The function $D_\mu$ is called a depth function \cite{zuo2000general} because it provides a measure of centrality relative to $\mu$.
Roughly speaking, $D_\mu(x)$ is the amount of mass that can be separated from $x$ by a hyperplane.
Therefore, the aforementioned results indicate that under a shape constraint on the measure $\mu$, even in large dimensions, there exist deep points.
Deep points are relevant in various applications:
In statistics, a deepest point (called Tukey median) provides a notion of center of a distribution that is robust to perturbations of that distribution \cite[Section 3.2.7]{SurveyDepth}, which is important when dealing with data from that distribution, that may have been corrupted. In numerical optimization, existence of deep points is essential for cutting plane methods \cite{bubeck2015convex} while Gr\"unbaum's inequality has also found applications in further convex optimization methods \cite{BV}. On the computational side, finding deep points is relevant in algorithmic geometry \cite{de2019discrete}.

In a non-Euclidean setup, a negative result was proved in \cite{rusciano2018riemannian}, showing that the above inequalities cannot be extended to nonpositively curved spaces in general.
Precisely, \cite[Theorem 2]{rusciano2018riemannian} states that given any Hadamard manifold $M$, for any probability measure $\mu$ on $M$ that is absolutely continuous with respect to the Riemannian volume measure (note that this additional restriction is only technical), there exists $x^* \in M$ such that any closed halfspace $H$ containing $x^*$ must satisfy $\mu(H) \ge 1/(n+1)$, where $n=\dim M$.
Moreover, there are cases where $\mu$ is the uniform distribution on a convex set and the bound is tight.
In this context, a closed halfspace is a subset of $M$ of the form $\{y \in M \mid \langle v,\dot{\gamma}_{xy}(0) \rangle \ge 0 \}$, for some $x\in M$ and $v \in T_x M\setminus\{0\}$, where $\gamma_{xy}$ denotes the (unique) minimal geodesic from $x$ to $y$.

In this article, we show that under a right framework, the above inequalities can be extended to non-Euclidean setups.
We work on metric measure spaces whose generalized Ricci curvature, in a synthetic sense, is nonnegative.
We appeal to Cheeger--Gromoll-type splitting theorems, which allow, as in the Euclidean case, to reduce the computations to a one-dimensional analysis. 

\subsection{Notations and definitions}

We briefly recall some concepts necessary to explain our results.

\subsubsection{Metric geometry.}

Let $(X,\dist_X)$ be a metric space.
Given $x,y\in X$, a (\emph{minimal}) \emph{geodesic} from $x$ to $y$ means a path $\gamma\colon [0,1]\lra X$ such that $\gamma(0)=x$, $\gamma(1)=y$, and $\dist_X(\gamma(s),\gamma(t))=|s-t|\dist_X(x,y)$ for all $s,t\in [0,1]$.
We call $(X,\dist_X)$ a \emph{geodesic space} if any pair $x,y\in X$ can be connected by a geodesic.
A subset $\Omega \subset X$ is said to be (\emph{geodesically}) \emph{convex} if, for any $x,y\in \Omega$, any geodesic between them is included in $\Omega$.
We say that a function $f\colon X\lra \R \cup \{\infty\}$ is \emph{convex} if it is convex along all geodesics, i.e., $f(\gamma(t))\leq (1-t)f(x)+tf(y)$ for all $x,y\in X$, all geodesics $\gamma\colon [0,1] \lra X$ from $x$ to $y$ and all $t\in [0,1]$.
In particular, $\supp(f):=f^{-1}(\R)$ is a convex set.
We say that $f$ is \emph{concave} if $-f$ is convex.

A \emph{straight line} is a map $\gamma\colon \R\lra X$ that satisfies $\dist_X(\gamma(s),\gamma(t))=|s-t|$ for all $s,t \in \R$.
Then, the \emph{Busemann function} associated with $\gamma$ is defined as
\begin{equation}\label{eq:Buse}
\bb_\gamma(x) :=\lim_{t\to\infty} \bigl\{ t-\dist_X \bigl( x,\gamma(t) \bigr) \bigr\}, \quad x \in X.
\end{equation}
The function $\bb_{\gamma}$ is $1$-Lipschitz and can be interpreted as a projection onto $\gamma$.
For instance, in the Euclidean case,  we have $\bb_\gamma(x)=\langle v,x-x_0 \rangle$, where $\gamma(t)=x_0+tv$ for some point $x_0$ and unit vector $v$.

We denote by $\mathcal{P}(X)$ the set of Borel probability measures on $X$ and, for $p \in [1,\infty)$, by $\mathcal{P}^p(X)$ the subset consisting of probability measures with finite $p$-th moment, i.e., those for which the function $\dist_X^p(\cdot,x_0)$ is integrable for some (and hence, all) $x_0\in X$.
For $\mu\in \mathcal P^2(X)$, a point $x_0 \in X$ attaining
\[
\inf_{z \in X} \int_X \dist_X^2(z,x) \,\mu(\diff x)
\]
is called a \emph{barycenter} of $\mu$.
More generally, even when $\mu$ only has finite first moment, we can define its barycenter as a point achieving
\[
\inf_{z \in X} \int_X \bigl\{ \dist_X^2(z,x) -\dist_X^2(z_0,x) \bigr\} \,\mu(\diff x),
\]
where $z_0 \in X$ is an arbitrarily fixed point.
In Euclidean spaces, we have the unique barycenter $\int_{\R^n} x \,\mu(\diff x)$.

\subsubsection{Curvature-dimension conditions.}

The curvature-dimension condition for metric measure spaces is a synthetic geometric notion of lower Ricci curvature bound described with the help of optimal transport theory.
For brevity, we consider only the case of nonnegative curvature.

A metric measure space $(X,\dist_X,\fm)$ will mean a complete separable metric space $(X,\dist_X)$ equipped with a Borel measure $\fm$ with $\fm(U) \in (0,\infty)$ for each nonempty bounded open set $U \subset X$.

Given $\nu_0,\nu_1 \in \mathcal{P}^2(X)$, the \emph{$L^2$-Kantorovich--Wasserstein distance} is defined by
\[
W_2(\nu_0,\nu_1) :=\inf_{\pi} \biggl( \int_{X \times X} \dist_X^2(x,y) \,\pi(\diff x \diff y) \biggr)^{1/2},
\]
where $\pi$ runs over all couplings of $\nu_0$ and $\nu_1$
(i.e., $\pi \in \mathcal{P}(X \times X)$ with marginals $\nu_0$ and $\nu_1$).
A geodesic $(\nu_{\lambda})_{\lambda \in [0,1]}$ with respect to $W_2$ is regarded as an \emph{optimal transport} from $\nu_0$ to $\nu_1$.

For $\nu =\zeta \fm \in \mathcal{P}(X)$ absolutely continuous with respect to $\fm$, we define the \emph{relative entropy}
\[
S_{\infty}(\nu) :=\int_X \zeta \log\zeta \diff\fm
\]
($S_\infty(\nu):=\infty$ if $\int_{\{\zeta>1\}} \zeta\log\zeta \diff\fm =\infty$),
and the \emph{R\'enyi entropy}
\[
S_N(\nu) :=\begin{cases}
-\int_X \zeta^{(N-1)/N} \diff\fm & N \in (1,\infty), \\
\int_X \zeta^{(N-1)/N} \diff\fm & N \in (-\infty,0).
\end{cases}
\]

We say that a metric measure space $(X,\dist_X,\fm)$ satisfies the \emph{curvature-dimension condition} $\CD(0,N)$ (or $(X,\dist_X,\fm)$ is a \emph{$\CD(0,N)$-space}) if the corresponding entropy $S_N$ is convex with respect to $W_2$ in the sense that,
for any absolutely continuous measures $\nu_0,\nu_1 \in \mathcal{P}^2(X)$, there exists a geodesic $(\nu_\lambda)_{\lambda \in [0,1]}$ between them with respect to $W_2$ such that
\begin{equation}\label{eqn:CD(0,N)}
S_N(\nu_\lambda) \le (1-\lambda)S_N(\nu_0) +\lambda S_N(\nu_1)
\end{equation}
holds for all $\lambda \in [0,1]$.
More generally, one can define $\CD(K,N)$-spaces for any $K\in \R$, where the coefficients $1-\lambda$ and $\lambda$ in \eqref{eqn:CD(0,N)} are replaced with distorsion coefficients that depend on $\lambda$, $K$ and $N$ \cite{SturmII}.

Consider an $n$-dimensional Riemannian manifold $(M,g)$ endowed with a measure $\fm=\e^{-\psi}\,\mathrm{vol}_g$ for a smooth function $\psi \in C^{\infty}(M)$, where $\mathrm{vol}_g$ is the volume measure induced from $g$.
The \emph{weighted Ricci curvature} (a.k.a.\ \emph{Bakry--\'Emery--Ricci curvature}) of the weighted Riemannian manifold $(M,g,\fm)$ is defined by
\begin{equation}\label{eq:Ric_N}
\Ric_N(v) :=\Ric_g(v) +\mathrm{Hess}\,\psi(v,v) -\frac{\langle \nabla \psi,v \rangle^2}{N-n}
\end{equation}
for $v \in TM$ and $N \in (-\infty,0] \cup (n,\infty)$ ($\Ric_g$ is the usual Ricci curvature of $g$).
We also define $\Ric_\infty$ and $\Ric_n$ as the limits. 
By definition, we have the monotonicity
\begin{equation}\label{eq:mono}
\Ric_n \le \Ric_N \le \Ric_{\infty} \le \Ric_{N'} \le \Ric_0
\end{equation}
for $n<N<\infty$ and $-\infty<N'<0$ ($\Ric_\infty$ can be also regarded as $\Ric_{-\infty}$).
Thus, for example, $\Ric_{N'} \ge 0$ is a weaker condition than $\Ric_N \ge 0$.
Note also that $\Ric_n=\lim_{N \downarrow n}\Ric_N \ge 0$ can make sense only when $\psi$ is constant.

A weighted Riemannian manifold $(M,g,\fm)$ is a $\CD(0,N)$-space if and only if the weighted Ricci curvature $\Ric_N$ is nonnegative \cite{CMS01,CMS06,LottVillani,Oneg,Oneedle,vRS,SturmI,SturmII}.

Moreover, the equivalence between $\Ric_N \ge 0$ and $\CD(0,N)$ also holds true for Finsler manifolds \cite{Oint}.
Then, to develop a genuinely Riemannian theory, the \emph{Riemannian curvature-dimension condition} $\RCD(0,N)$ was introduced as the combination of $\CD(0,N)$ and the so-called \emph{infinitesimal Hilbertianity} (or, equivalently, the linearity of heat flow) \cite{AGS,EKS,gsplit}.
In $\RCD(0,N)$-spaces, we can obtain much finer properties including a splitting theorem discussed below.
We refer to \cite{Sturm24} for a recent survey.

\subsubsection{Splitting theorems.}

For a Riemannian manifold $(M,g)$ of nonnegative Ricci curvature, Cheeger--Gromoll's celebrated \emph{splitting theorem} \cite{CG} asserts that, if there is a straight line $\gamma\colon \R \lra M$, then $M$ is isometric to a product space $\R \times \Sigma$, where $\Sigma$ is a Riemannian manifold of nonnegative Ricci curvature, and the Busemann function $\bb_\gamma$ as in \eqref{eq:Buse} coincides with the projection to $\R$.
The splitting theorem was generalized to $\RCD(0,N)$-spaces by Gigli \cite[Theorem 1.4]{gsplit},\cite{Greview}.
In short, it states that if $(X,\dist_X,\fm)$ is an $\RCD(0,N)$-space for some $N\in (1,\infty)$ including a straight line $\gamma\colon \R\lra X$, then $X$ is isometric to a product space $\R \times Y$ where $(Y,\dist_Y,\fn)$ is an $\RCD(0,N-1)$-space when $N \ge 2$, and $Y$ is a singleton when $N \in (1,2)$.
(We will not consider the case of $N=1$, since $\RCD(0,N)$ with $N>1$ is weaker than $\RCD(0,1)$.)

In general, such an isometric splitting is false for $\CD(0,N)$-spaces, unless the infinitesimal Hilbertianity is assumed (see \eqref{prob:B} in Section~\ref{sc:outro}).
This is why, in our main results, we consider only $\RCD$-spaces, although some intermediate results may be stated in more generality for $\CD$-spaces.

\subsection{Main results}

\begin{theorem}[Main theorem; $N>1$]\label{thm:main}
Let $(X,\dist_X,\fm)$ be an $\RCD(0,N)$-space with $N \in (1,\infty)$, $\mu=\rho\fm \in \mathcal{P}^1(X)$ with a measurable function $\rho\colon X \lra [0,\infty)$, and $x_0 \in X$ be any barycenter of $\mu$.
Suppose that there is a straight line $\gamma\colon \R \lra X$.

\begin{enumerate}[{\rm (i)}]
\item\label{main_pos}
If $\rho^{1/(\beta-N)}$ is concave on $\rho^{-1}((0,\infty))$ for some $\beta>N$, then the Busemann function $\bb_{\gamma}\colon X \lra \R$ satisfies
\begin{align}
\mu \bigl( \{ x \in X \mid \bb_{\gamma}(x) \le \bb_{\gamma}(x_0) \} \bigr)
&\ge \biggl( \frac{\beta}{\beta +1} \biggr)^\beta, 
\label{eq:main}\\
\mu \bigl( \{ x \in X \mid \bb_{\gamma}(x) \ge \bb_{\gamma}(x_0) \} \bigr)
&\ge \biggl( \frac{\beta}{\beta +1} \biggr)^\beta.
\nonumber
\end{align}

\item\label{main_inf}
If $\log\rho\colon X \lra \R \cup \{-\infty\}$ is concave, then $\bb_\gamma$ satisfies
\begin{align*}
\mu \bigl( \{ x \in X \mid \bb_{\gamma}(x) \le \bb_{\gamma}(x_0) \} \bigr) &\ge \e^{-1}, \\
\mu \bigl( \{ x \in X \mid \bb_{\gamma}(x) \ge \bb_{\gamma}(x_0) \} \bigr) &\ge \e^{-1}.
\end{align*}

\item\label{main_neg}
If $\rho^{1/(\beta -N)}\colon X \lra \R \cup \{\infty\}$ is convex for some $\beta<-1$, then we have
\begin{align*}
\mu \bigl( \{ x \in X \mid \bb_{\gamma}(x) \le \bb_{\gamma}(x_0) \} \bigr)
&\ge \biggl( \frac{\beta}{\beta +1} \biggr)^\beta, \\
\mu \bigl( \{ x \in X \mid \bb_{\gamma}(x) \ge \bb_{\gamma}(x_0) \} \bigr)
&\ge \biggl( \frac{\beta}{\beta +1} \biggr)^\beta.
\end{align*}
\end{enumerate}
\end{theorem}

\begin{remark}\label{rm:main}
Under the hypothesis in \eqref{main_pos}, it follows from Lemma~\ref{lm:beta} that $(X,\dist_X,\mu)$ is an $\RCD(0,\beta)$-space.
Then, by virtue of Lemma~\ref{lm:bdd}, the support of $\mu$ is necessarily bounded.
In particular, $\mu$ is of finite second moment.
\end{remark}

The condition on $\rho$ in \eqref{main_pos} is equivalent to the concavity of $\tilde{\rho}$ by setting $\tilde{\rho}(x):=\rho(x)^{1/(\beta -N)}$ if $\rho(x)>0$ and $\tilde{\rho}(x):=-\infty$ if $\rho(x)=0$.
Note that, on the one hand, $\supp(\mu)$ is convex in all the cases \eqref{main_pos}--\eqref{main_neg}.
On the other hand, given a convex set $\Omega\subset X$ with $\fm(\Omega) \in (0,\infty)$, the uniform distribution $\mu_{\Omega} :=\fm(\Omega)^{-1} \chi_\Omega \,\fm$ on $\Omega$ satisfies the hypothesis of \eqref{main_pos} for any $\beta>N$, where $\chi_\Omega$ is the indicator function of $\Omega$ (with value $1$ on $\Omega$ and $0$ on $X \setminus \Omega$).
Hence, taking the limit as $\beta \downarrow N$ yields the following corollary, as a direct extension of Gr\"{u}nbaum's inequality.
By a \emph{barycenter} of $\Omega$, we will mean a barycenter of $\mu_{\Omega}$.

\begin{corollary}\label{cr:main}
Let $\Omega$ be a convex set in an $\RCD(0,N)$-space $(X,\dist_X,\fm)$ with $N \in (1,\infty)$ such that $\fm(\Omega) \in (0,\infty)$,
and let $x_0 \in X$ be a barycenter of $\Omega$.
Then, for any straight line $\gamma\colon \R \lra X$,
the associated Busemann function $\bb_{\gamma}\colon X \lra \R$ satisfies
\begin{align*}
\fm \bigl( \{ x \in \Omega \mid \bb_{\gamma}(x) \le \bb_{\gamma}(x_0) \} \bigr)
&\ge \biggl( \frac{N}{N+1} \biggr)^N \cdot \fm(\Omega), \\
\fm \bigl( \{ x \in \Omega \mid \bb_{\gamma}(x) \ge \bb_{\gamma}(x_0) \} \bigr)
&\ge \biggl( \frac{N}{N+1} \biggr)^N \cdot \fm(\Omega).
\end{align*}
\end{corollary}

We remark that a barycenter of $\Omega$ may not be unique in this generality.
For instance, consider the cylinder $X=\R \times \Sph^1$ endowed with the $2$-dimensional Hausdorff measure, which is an $\RCD(0,2)$-space.
Then, for $\Omega=[-1,1] \times \Sph^1$, we find that any point on the circle $\{0\} \times \Sph^1$ is a barycenter of $\Omega$.
It seems unclear (to the authors) if every barycenter of $\Omega$ lives in $\Omega$.

\begin{remark}\label{rm:horo}
A subset of the form $\bb_\gamma^{-1}((-\infty,r])$, for some straight line $\gamma$ and $r \in \R$, is called a (closed) \emph{horoball}.
Thus, Corollary~\ref{cr:main} can be rephrased by saying that every horoball (or the closure of the complement of a horoball) including $x_0$ has mass at least $(N/(N+1))^N \cdot \fm(\Omega)$.
In Euclidean spaces, horoballs are simply closed halfspaces.
In general, horoballs are not geodesically convex, unless $X$ is nonpositively curved.
See, for example, \cite{LLN} concerning convex optimization on Hadamard spaces by means of horoballs.
\end{remark}

We also study when equality holds in \eqref{eq:main}.
Roughly speaking, equality holds only when $\mu$ has a cone structure (see Theorem~\ref{th:rigid} for the precise statement).
This kind of \emph{rigidity} for geometric and analytic inequalities is one of the major problems in comparison geometry and geometric analysis (see, e.g., \cite{GKKO,OT}).
Cavalletti--Mondino's \emph{localization} (also called \emph{needle decomposition}) \cite{CM}, together with a detailed one-dimensional analysis, plays a crucial role in our rigidity result.
We can even consider the \emph{stability} problem in a similar way
(see Section~\ref{sc:stab}).
We refer to \cite{Gr00} for a stability result concerning Gr\"unbaum's inequality in the Euclidean setting, in terms of the volume of the symmetric difference from a cone.

The main ingredients of the proofs of our results are
Gigli's splitting theorem for $\RCD(0,N)$-spaces and Cavalletti--Mondino's localization for essentially non-branching $\CD(K,N)$-spaces as we mentioned above.
In fact, the formulation of our results using a straight line is strongly inspired by the splitting theorem.
Both these ingredients are valid only for $N \in (1,\infty)$ in this generality.

Nonetheless, in the smooth setting of weighted Riemannian manifolds, the isometric splitting is known by Lichnerowicz, Fang--Li--Zhang \cite{Li,FLZ} ($N=\infty$) and Wylie \cite{Wy} ($N \in (-\infty,1)$),
and the localization is also available by Klartag \cite{Kl}.
Thus, we have the following counterparts to Theorem~\ref{thm:main} and Corollary~\ref{cr:main}.


\begin{theorem}[Main theorem; $N=\infty$, $N<-1$]\label{th:neg}
Let $(M,g,\fm)$, $\fm=\e^{-\psi} \,\mathrm{vol}_g$, be a complete weighted Riemannian manifold of $\Ric_N \ge 0$ for $N=\infty$ or $N \in (-\infty,-1)$, where $\psi \in C^2(M)$ is bounded from above,
and $\mu=\rho\fm \in \mathcal{P}^1(M)$ with $\rho \colon M \lra [0,\infty)$.

\begin{enumerate}[{\rm (i)}]
\item\label{main2_inf}
Suppose that $N=\infty$ and $\log\rho \colon M \lra \R \cup \{-\infty\}$ is concave.
Then, for any barycenter $x_0 \in M$ of $\mu$ and any straight line $\gamma\colon \R \lra M$,
the associated Busemann function $\bb_{\gamma}\colon M \lra \R$ satisfies
\begin{align}
\mu \bigl( \{ x \in M \mid \bb_{\gamma}(x) \le \bb_{\gamma}(x_0) \} \bigr) &\ge \e^{-1}, \label{eq:main2_inf}\\
\mu \bigl( \{ x \in M \mid \bb_{\gamma}(x) \ge \bb_{\gamma}(x_0) \} \bigr) &\ge \e^{-1}. \nonumber
\end{align}

\item\label{main2_neg}
Suppose that $N \in (-\infty,-1)$ and $\rho^{1/(\beta -N)}$ is concave on $\rho^{-1}((0,\infty))$ for some $\beta \in (N,-1)$.
Then, for $x_0 \in M$ and $\bb_{\gamma}$ as above, we have
\begin{align}
\mu \bigl( \{ x \in M \mid \bb_{\gamma}(x) \le \bb_{\gamma}(x_0) \} \bigr)
&\ge \biggl( \frac{\beta}{\beta +1} \biggr)^\beta, \label{eq:main2_neg}\\
\mu \bigl( \{ x \in M \mid \bb_{\gamma}(x) \ge \bb_{\gamma}(x_0) \} \bigr)
&\ge \biggl( \frac{\beta}{\beta +1} \biggr)^\beta. \nonumber
\end{align}
\end{enumerate}
\end{theorem}

We remark that the upper boundedness of $\psi$ is assumed for applying the splitting theorem (see Remark~\ref{rm:psi}).

\begin{corollary}\label{cr:neg}
Let $(M,g,\fm)$ be as in Theorem~$\ref{th:neg}$,
and $\Omega \subset M$ be a convex set such that $\fm(\Omega) \in (0,\infty)$ and $\fm(\Omega)^{-1} \chi_\Omega\,\fm \in \mathcal{P}^1(M)$.
Then, for any barycenter $x_0 \in X$ of $\Omega$ and any straight line $\gamma\colon \R \lra X$,
the associated Busemann function $\bb_{\gamma}\colon X \lra \R$ satisfies
\begin{align*}
\fm \bigl( \{ x \in \Omega \mid \bb_{\gamma}(x) \le \bb_{\gamma}(x_0) \} \bigr)
&\ge \e^{-1} \cdot \fm(\Omega), \\
\fm \bigl( \{ x \in \Omega \mid \bb_{\gamma}(x) \ge \bb_{\gamma}(x_0) \} \bigr)
&\ge \e^{-1} \cdot \fm(\Omega)
\end{align*}
when $N=\infty$, and
\begin{align*}
\fm \bigl( \{ x \in \Omega \mid \bb_{\gamma}(x) \le \bb_{\gamma}(x_0) \} \bigr)
&\ge \biggl( \frac{N}{N+1} \biggr)^N \cdot \fm(\Omega), \\
\fm \bigl( \{ x \in \Omega \mid \bb_{\gamma}(x) \ge \bb_{\gamma}(x_0) \} \bigr)
&\ge \biggl( \frac{N}{N+1} \biggr)^N \cdot \fm(\Omega)
\end{align*}
for $N \in (-\infty,-1)$.
\end{corollary}

The rest of the article is organized as follows.
In Section~\ref{sc:prel}, we review some necessary properties of $\RCD(0,N)$-spaces.
We also illustrate our results in Euclidean spaces, recovering known results and obtaining new ones.
Section~\ref{sc:1d} is devoted to the one-dimensional analysis, which plays a central role in the following discussions.
In Section~\ref{sc:proof}, we prove Theorem~\ref{thm:main} and also study the rigidity problem.
By a similar analysis, in Section~\ref{sc:neg}, we prove Theorem~\ref{th:neg} as well as the corresponding rigidity result.
We give some stability results in Section~\ref{sc:stab},
and close the article with several further problems in Section~\ref{sc:outro}.

\section{Preliminaries}\label{sc:prel}

\subsection{Properties of $\RCD(0,N)$-spaces}

The next lemma may be a well known fact (cf.\ \cite[Theorem~7]{Ya}), but we could not find in the literature.

\begin{lemma}\label{lm:bdd}
Let $(X,\dist_X,\fm)$ be a $\CD(0,N)$-space with $N \in (1,\infty)$.
If $\fm(X)<\infty$, then the diameter of $X$ is finite.
\end{lemma}

\begin{proof}
The proof is based on the following Bishop--Gromov volume comparison theorem \cite[Theorem 2.3]{SturmII}: 
\begin{equation}\label{eq:BG}
\frac{\fm(B(x,R))}{\fm(B(x,r))} \le \biggl( \frac{R}{r} \biggr)^N
\quad\ \text{for all}\ x \in X,\ 0<r<R,
\end{equation}
where $B(x,r)$ denotes the open ball with center $x$ and radius $r$.
Fix $x_0 \in X$ and assume on the contrary that there is a sequence $(x_k)_{k \ge 1}$ in $X$ such that $\dist_X(x_0,x_k) \ge k$ for all $k\geq 1$.
Then we infer from \eqref{eq:BG} that, for $k \ge 2$,
\[ \fm\Bigl( B\bigl( x_k,\dist_X(x_0,x_k)+1 \bigr) \Bigr)
\le \biggl( \frac{\dist_X(x_0,x_k)+1}{\dist_X(x_0,x_k)-1} \biggr)^N \fm\Bigl( B\bigl( x_k,\dist_X(x_0,x_k)-1 \bigr) \Bigr). \]
Thus, we deduce that
\begin{align*}
\fm\bigl( B(x_0,1) \bigr)
&\le \fm\Bigl( B\bigl( x_k,\dist_X(x_0,x_k)+1 \bigr) \Bigr) -\fm\Bigl( B\bigl( x_k,\dist_X(x_0,x_k)-1 \bigr) \Bigr) \\
&\le \biggl\{ \biggl( \frac{\dist_X(x_0,x_k)+1}{\dist_X(x_0,x_k)-1} \biggr)^N -1 \biggr\} \cdot \fm\Bigl( B\bigl( x_k,\dist_X(x_0,x_k)-1 \bigr) \Bigr) \\
&\le \biggl\{ \biggl( \frac{\dist_X(x_0,x_k)+1}{\dist_X(x_0,x_k)-1} \biggr)^N -1 \biggr\} \cdot \fm(X) \\
&\to 0
\end{align*}
as $k \to \infty$.
This implies $\fm(B(x_0,1))=0$, a contradiction.
\end{proof}

In particular, every convex set $\Omega \subset X$ with $\fm(\Omega) \in (0,\infty)$ (as in Corollary~\ref{cr:main}) is bounded,
since $(\Omega,\dist_X|_{\Omega},\fm|_{\Omega})$ is again a $\CD(0,N)$-space.
The above lemma can be applied to the space $(X,\dist_X,\mu)$ in Theorem~\ref{thm:main}\eqref{main_pos} thanks to the following fact.

\begin{lemma}\label{lm:beta}
Let $(X,\dist_X,\fm)$ be an $\RCD(0,N)$-space with $N \in (1,\infty)$,
and $\mu=\rho \fm$ be a measure with $\rho \colon X \lra [0,\infty)$
such that $\mu(X)>0$ and $\rho^{1/(\beta -N)}$ is concave on $\rho^{-1}((0,\infty))$ for some $\beta>N$.
Then, $(\supp(\mu),\dist_X|_{\supp(\mu)},\mu)$ is an $\RCD(0,\beta)$-space.
\end{lemma}

\begin{proof}
We give an outline of the proof for completeness
(we refer to \cite[Theorem~1.7(ii)]{SturmII} for the smooth case).
Note that it is sufficient to show that $(\supp(\mu),\dist_X|_{\supp(\mu)},\mu)$ satisfies $\CD(0,\beta)$.
Since $(X,\dist_X)$ is non-branching by \cite[Theorem 1.3]{deng2020},
in view of \cite[Proposition~4.2]{SturmII},
$(X,\dist_X,\fm)$ being a $\CD(0,N)$-space is equivalent to the concavity of $\lambda \longmapsto \zeta_\lambda (\eta(\lambda))^{-1/N}$ for almost every geodesic $\eta\colon [0,1] \lra X$ along which an $L^2$-optimal transport $(\nu_\lambda)_{\lambda \in [0,1]}$ is done,
where $\nu_\lambda =\zeta_\lambda \fm$.

When we consider $\mu=\rho\fm$ instead of $\fm$, the corresponding density function of $\nu_\lambda$ becomes $\zeta_\lambda \rho^{-1}$, provided $\supp(\nu_{\lambda}) \subset \supp(\mu)$.
Indeed, if $\supp(\nu_0) \cup \supp(\nu_1) \subset \supp(\mu)$, then $\supp(\nu_{\lambda}) \subset \supp(\mu)$ by the convexity of $\supp(\mu)$.
We deduce from the concavity of $\rho^{1/(\beta -N)}$ and the H\"older inequality that
\begin{align*}
[\zeta_\lambda \rho^{-1}] \bigl( \eta(\lambda) \bigr)^{-1/\beta}
&= \Bigl( \zeta_\lambda \bigl( \eta(\lambda) \bigr)^{-1/N} \Bigr)^{N/\beta}
 \Bigl( \rho\bigl( \eta(\lambda) \bigr)^{1/(\beta -N)} \Bigr)^{(\beta -N)/\beta} \\
&\ge \Bigl( (1-\lambda)\zeta_0 \bigl( \eta(0) \bigr)^{-1/N} +\lambda\zeta_1 \bigl( \eta(1) \bigr)^{-1/N} \Bigr)^{N/\beta} \\
&\quad\times \Bigl( (1-\lambda)\rho\bigl( \eta(0) \bigr)^{1/(\beta -N)} +\lambda\rho\bigl( \eta(1) \bigr)^{1/(\beta -N)} \Bigr)^{(\beta -N)/\beta} \\
&\ge (1-\lambda)[\zeta_0 \rho^{-1}] \bigl( \eta(0) \bigr)^{-1/\beta}
 +\lambda[\zeta_1 \rho^{-1}] \bigl( \eta(1) \bigr)^{-1/\beta}.
\end{align*}
Therefore, $\lambda \longmapsto [\zeta_\lambda \rho^{-1}](\eta(\lambda))^{-1/\beta}$ is concave, and $(\supp(\mu),\dist_X|_{\supp(\mu)},\mu)$ is an $\RCD(0,\beta)$-space.
\end{proof}

Next we summarize some key ingredients for the proof of Theorem~\ref{thm:main}, based on Gigli's splitting theorem and an observation related to the lemma above.
We remark that, though Lemma~\ref{lm:beta} could be also extended to include the cases of $\beta=\infty$ and $\beta<-1$ as in the next proposition, we restricted ourselves to $\beta>N$ for simplicity.

\begin{proposition}\label{pr:proj}
Let $(X,\dist_X,\fm)$ be an $\RCD(0,N)$-space for some $N\in (1,\infty)$ and $\gamma\colon \R \lra X$ be a straight line.
Then, there exists an $\RCD(0,N-1)$-space $(Y,\dist_Y,\fn)$ and an isometry $T\colon (X,\dist_X) \lra (\R\times Y,d)$ such that
\begin{itemize}
    \item $d((s,y),(t,z))=\left(|s-t|^2+\dist_Y(y,z)^2\right)^{1/2}$ for all $s,t\in\R$ and $y,z\in Y;$
    \item $\Pi_{\R}(T(x))=\bb_\gamma(x)$ for all $x \in X$, where $\Pi_{\R}\colon \R \times Y \lra \R$ is the projection$;$
    \item $T\#\fm=\diff x\otimes \fn$, where $T\#\fm$ denotes the push-forward of $\fm$ by $T$ and $\diff x$ is the Lebesgue measure on $\R$.
\end{itemize}
Moreover, let $\mu=\rho\fm$ be a probability measure on $X$ for some measurable function $\rho\colon X \lra [0,\infty)$ and put $\bb_\gamma\#\mu=w\diff x$.
\begin{enumerate}[{\rm (i)}]
\item\label{key_pos}
If $\rho^{1/(\beta-N)}$ is concave on $\rho^{-1}((0,\infty))$ for some $\beta>N$, then $w^{1/(\beta-1)}$ is concave on $w^{-1}((0,\infty))$.
\item\label{key_inf}
If $\log\rho\colon X \lra \R \cup \{-\infty\}$ is concave, then $\log w$ is concave.
\item\label{key_neg}
If $\rho^{1/(\beta-N)}\colon X \lra \R \cup \{\infty\}$ is convex for some $\beta<1$, then $w^{1/(\beta-1)}$ is convex.
\end{enumerate}
\end{proposition}

\begin{proof}
The first part of the proposition comes from \cite[Theorem~1.4]{gsplit}.
Then, using the product structure $X=\R \times Y$, $w$ is explicitly given by
\[ w(t) =\int_Y \rho(t,y) \,\fn(\diff y). \]

\eqref{key_pos}
Though we give a direct proof, one can also reduce the concavity of $w^{1/(\beta -1)}$ to Lemma~\ref{lm:beta}, for it is equivalent to $\CD(0,\beta)$ of the interval $(\supp(w),|\cdot|,w \diff x)$ (see the beginning of Section~\ref{sc:1d}).
Take $a<b$ and $\lambda \in (0,1)$, and define functions $f,g,h \colon Y \lra \R$ by
\[
h(y) := \rho\bigl( (1-\lambda)a +\lambda b,y \bigr), \qquad
f(y) := \rho(a,y), \qquad
g(y) := \rho(b,y).
\]
For any minimal geodesic $\eta\colon [0,1] \lra Y$ in $(Y,\dist_Y)$, observe that the curve
\[
\lambda \,\longmapsto\, \bigl( (1-\lambda)a +\lambda b,\eta(\lambda) \bigr) \in \R \times Y =X
\]
is also a minimal geodesic.
Thus, the assumed concavity of $\rho^{1/(\beta-N)}$ implies
\[
 h\bigl( \eta(\lambda) \bigr)^{1/(\beta-N)} \ge (1-\lambda) f\bigl( \eta(0) \bigr)^{1/(\beta-N)} +\lambda g\bigl( \eta(1) \bigr)^{1/(\beta-N)}.
\]
Now, in order to apply the Borell--Brascamp--Lieb inequality on $(Y,\dist_Y,\fn)$ in \cite[Theorem 3.1]{bacher2010borell}, one needs to check that $(Y,\dist_Y,\fn)$ is non-branching, which is the case thanks to \cite{deng2020}.
Hence, we obtain from the Borell--Brascamp--Lieb inequality $\mathrm{BBL}(0,N-1)$ with parameter $p=1/(\beta -N)>0$ (as in \cite[Definition~1.1]{bacher2010borell}) that, since $p/(1+(N-1)p)=1/(\beta -1)$,
\[
\int_Y h \diff\fn
    \ge \Biggl( (1-\lambda) \biggl( \int_Y f \diff\fn \biggr)^{1/(\beta -1)} +\lambda \biggl( \int_Y g \diff\fn \biggr)^{1/(\beta -1)} \Biggr)^{\beta -1}.
\]
This yields the concavity of $w^{1/(\beta -1)}$.

\eqref{key_inf}
We similarly find
\[
 \log h\bigl( \eta(\lambda) \bigr) \ge (1-\lambda) \log f\bigl( \eta(0) \bigr) +\lambda \log g\bigl( \eta(1) \bigr).
\]
Then, $\mathrm{BBL}(0,N-1)$ with $p=0$ (in other words, the Pr\'ekopa--Leindler inequality) shows the claim
\[
\int_Y h \diff\fn \ge \biggl( \int_Y f \diff\fn \biggr)^{1-\lambda} \biggl( \int_Y g \diff\fn \biggr)^\lambda.
\]

\eqref{key_neg}
In this case, we have
\[
 h\bigl( \eta(\lambda) \bigr) \ge \Bigl( (1-\lambda) f\bigl( \eta(0) \bigr)^{1/(\beta-N)} +\lambda g\bigl( \eta(1) \bigr)^{1/(\beta-N)} \Bigr)^{\beta -N},
\]
then the claim follows from $\mathrm{BBL}(0,N-1)$ with $p=1/(\beta -N) >-1/(N-1)$.
\end{proof}

\subsection{Euclidean case}

In this subsection, we describe our results in the Euclidean setting.
We first remark that the curvature-dimension condition is intimately related to Borell's $s$-concavity \cite{Borell}.
For $s \in \R \cup \{\pm\infty\}$, $\lambda \in (0,1)$ and $a,b\geq 0$, the \emph{$s$-mean} is defined by
\[
\mathcal{M}_s(a,b;\lambda) :=\begin{cases}
\bigl( (1-\lambda)a^s +\lambda b^s \bigr)^{1/s} &\text{if } s>0, \text{ or if } s<0 \text{ and } ab>0 ,\\
0 &\text{if } s<0 \text{ and } ab=0, \\
a^{1-\lambda} b^\lambda &\text{if } s=0, \\
\min\{ a,b \} &\text{if } s=-\infty, \\
\max\{ a,b \} &\text{if } s=\infty.
\end{cases}
\]
We say that a positive Radon measure $\mu$ on $\R^n$ is \emph{$s$-concave} if, for any Borel measurable sets $A,B \subset \R^n$ and $\lambda \in (0,1)$,
\[
\mu\bigl( (1-\lambda)A+\lambda B \bigr)
\ge \mathcal{M}_s \bigl( \mu(A),\mu(B) ;\lambda \bigr),
\]
where $(1-\lambda)A+\lambda B:=\{ (1-\lambda)x+\lambda y \mid x \in A,\, y \in B \}$ is the Minkowski sum.
Then, Borell's classical result \cite[Theorem~3.2]{Borell} shows the following.

\begin{theorem}[Borell's theorem]\label{th:Borell}
\begin{itemize}
    \item If $s>1/n$, then there does not exist any $s$-concave measure on $\R^n$.
    \item If $s=1/n$, then the only $s$-concave measures on $\R^n$ are constant multiplications of the $n$-dimensional Lebesgue measure $\mathcal{L}^n$.
    \item For $s \in (0,1/n)$, a measure $\mu$ on $\R^n$ is $s$-concave if and only if it has a density function $\rho$ with respect to $\mathcal{L}^n$ and $\rho^p$ is concave with $p=s/(1-sn)$.
    \item For $s\in [-\infty,0]$, a measure $\mu$ on $\R^n$ is $s$-concave if and only if it has a density function $w$ with respect to the Lebesgue measure on an affine subspace $V$ and $w$ satisfies
    \begin{equation}\label{eq:s-con}
    w\bigl( (1-\lambda) x +\lambda y \bigr)
    \ge \mathcal{M}_q\bigl( w(x),w(y);\lambda \bigr)
    \end{equation}
    for all $x,y \in V$ and $\lambda \in (0,1)$, where $q:=s/(1-s\dim V)$ $($which should be understood as $-1/\dim V$ if $s=-\infty)$.
\end{itemize}
\end{theorem}

We remark that $\supp(\mu)=\R^n$ when $s \in (0,1/n]$, while $\supp(\mu)$ is a convex set in $V$ for $s \in [-\infty,0]$.

\begin{corollary}
    Let $s\in [-\infty,1/n]$ and $\mu$ be an $s$-concave measure on $\R^n$ with $\supp(\mu)=\R^n$.
    Then, for any affine subspace $V \subset \R^n$ and the orthogonal projection $\Pi\colon \R^n \lra V$,
    $\Pi\#\mu$ has a density function $w$ with respect to the Lebesgue measure on $V$, which satisfies \eqref{eq:s-con}. 
\end{corollary}

\begin{proof}
    It is readily seen that $\Pi\#\mu$ is $s$-concave,
    then we apply Theorem~\ref{th:Borell}.
\end{proof}

In particular, if $\dim V=1$, then $w$ satisfies \eqref{eq:s-con} with $q=s/(1-s)$, which corresponds to Proposition~\ref{pr:proj} in the Euclidean case (with $s=1/\beta$).

In view of Theorem~\ref{th:Borell},
as a particular case of Theorem~\ref{thm:main}, we obtain the following result.
Note that $s>-1$ since $\beta<-1$.

\begin{theorem} \label{thm:Eucl}
Let $s\in (-1,1/n]$ and $\mu \in \mathcal{P}^1(\R^n)$ be $s$-concave in its support.
Then, every closed halfspace $H$ containing the barycenter of $\mu$ satisfies $\mu(H) \geq (1/(1+s))^{1/s}$, understood as $\mu(H) \ge \e^{-1}$ if $s=0$.
\end{theorem}


Classical Grünbaum's inequality is recovered by Theorem~\ref{thm:Eucl} with $s=1/n$.
Moreover, as mentioned in the introduction, the case of log-concave distributions ($s=0$) was also known.

\begin{remark}\label{rm:BM}
The $s$-concavity is generalized to the \emph{Brunn--Minkowski inequality} $\mathrm{BM}(0,N)$, with $N=1/s$.
Precisely, the curvature-dimension condition $\CD(0,N)$ implies $\mathrm{BM}(0,N)$ (see \cite[Proposition~2.1]{SturmII} for $N \in (1,\infty)$, \cite[Theorem~30.7]{Villani} for $N=\infty$, \cite[Theorem~4.8]{Oneg} for $N<0$,  and \cite[(3.9)]{Oneedle} for $N=0$).
Moreover, for weighted Riemannian manifolds and $N\in (1,\infty)$, the converse implication can be also found in \cite[Theorem~1.1]{magnabosco2024brunn}.
\end{remark}


\section{One-dimensional analysis}\label{sc:1d}

In this section, as a key step for the proof of Theorem~\ref{thm:main}, we first state Gr\"{u}nbaum's inequality in the case of one-dimensional $\CD(0,N)$-spaces.

\subsection{One-dimensional $\CD(0,N)$-spaces}\label{ssc:1dCD}

We consider a one-dimensional space $((a,b),|\cdot|,\e^{-\psi}\diff x)$, where $-\infty \le a<b \le \infty$, $|\cdot|$ denotes the absolute value giving rise to the canonical distance on $\R$, $\diff x$ is the Lebesgue measure on $\R$, and $\psi \colon (a,b) \lra \R$ is a continuous function.
In this case, being a $\CD(0,N)$-space is equivalent to
\begin{equation}\label{eq:CD0N}
\psi'' -\frac{(\psi')^2}{N-1} \ge 0
\end{equation}
in the weak sense (recall the definition of $\Ric_N$ from \eqref{eq:Ric_N}).
By setting $w=\e^{-\psi}$, this is equivalent to the concavity of $w^{1/(N-1)}$ if $N \in (1,\infty)$, the concavity of $\log w$ if $N=\infty$, and to the convexity of $w^{1/(N-1)}$ if $N \in (-\infty,1)$.

When we assume that $\mu=w \diff x$ is a probability measure, on the one hand, for $N \in (1,\infty)$, observe from the concavity of $w^{1/(N-1)}$ that $\mu$ can only be supported on some bounded interval, thereby $a>-\infty$ and $b<\infty$.
On the other hand, for $N=\infty$ (resp.\ $N \in (-\infty,1)$), the concavity of $\log w$ (resp.\ convexity of $w^{1/(N-1)}$) does not imply such boundedness, by letting $\log w=-\infty$ (resp.\ $w^{1/(N-1)}=\infty$) outside the support of $\mu$.

For $N \in (1,\infty)$, the space $([0,\infty),|\cdot|,Nx^{N-1} \diff x)$ is a model $\CD(0,N)$-space, that is, it enjoys equality in \eqref{eq:CD0N}.
Moreover, normalized as a probability space, 
\begin{equation}\label{eq:model}
([0,1],|\cdot|,Nx^{N-1} \diff x)
\end{equation}
is a model space that also reaches the equality case in Gr\"unbaum's inequality.
Indeed, its barycenter is given by
\[
\int_0^1 x \cdot Nx^{N-1} \diff x = \frac{N}{N+1},
\]
and we have
\[
\int_0^{N/(N+1)} Nx^{N-1} \diff x = \left(\frac{N}{N+1} \right)^N.
\]
In the original Gr\"unbaum's inequality, this model corresponds to the projection of the uniform distribution over a finite cone along its axis.


\subsection{Case of $N \in (1,\infty)$}\label{ssc:1d_pos}

We first consider the case of $N \in (1,\infty)$.

\begin{lemma}[Gr\"unbaum's inequality on intervals; $N>1$]\label{lm:Grunbaum1d}
Let $a,b\in\R$ with $a<0<b$ and $w\colon (a,b) \lra [0,\infty)$ be a nonnegative function such that $w^{1/(N-1)}$ is concave for some $N>1$ and
\[
\int_a^b w(x) \diff x=1, \qquad \int_a^b xw(x)\diff x = 0.
\]
Then, we have
\begin{equation}\label{eq:1d_pos}
\int_a^0 w(x) \diff x \geq \biggl( \frac{N}{N+1} \biggr)^N, \qquad
\int_0^b w(x) \diff x \geq \biggl( \frac{N}{N+1} \biggr)^N.
\end{equation}
\end{lemma}

\begin{proof}
Let $R(x):=\int_a^x w(s) \diff s$ be the cumulative distribution function, which satisfies $0 \le R \le 1$, $R(a)=0$ and $R(b)=1$ by definition.
Note also that
\begin{equation}\label{eq:intR}
\int_a^b R(x) \diff x
=\Bigl[ xR(x) \Bigr]_a^b -\int_a^b xw(x) \diff x
=b.
\end{equation}

Since $w^{1/(N-1)}$ is concave, the Borell--Brascamp--Lieb inequality \cite{Borell,BL} with parameter $1/(N-1)$ implies that $R^{1/N}$ is concave.
Here we also give a detailed proof for later use in the rigidity problem.
Fix $x,y \in (a,b)$ and $\lambda \in (0,1)$.
For $t \in (0,1)$, we take $\sigma(t) \in (a,x)$ and $\tau(t) \in (a,y)$ satisfying
\[ \frac{1}{R(x)} \int_a^{\sigma(t)} w(s) \diff s = \frac{1}{R(y)} \int_a^{\tau(t)} w(s) \diff s =t. \]
By differentiating in $t$, we have
\begin{equation}\label{eq:u'}
\frac{w(\sigma(t)) \sigma'(t)}{R(x)} = \frac{w(\tau(t)) \tau'(t)}{R(y)} =1.
\end{equation}
We put $\theta(t) :=(1-\lambda) \sigma(t) +\lambda \tau(t)$ and observe
\[ R\bigl( (1-\lambda)x +\lambda y \bigr)
 = \int_a^{(1-\lambda)x +\lambda y} w(s) \diff s
 = \int_0^1 w\bigl( \theta(t) \bigr) \theta'(t) \diff t. \]
It follows from the concavity of $w^{1/(N-1)}$, \eqref{eq:u'} and the H\"older inequality that
\begin{align}
\begin{split} \label{eq:BBL}
\int_0^1 w\bigl( \theta(t) \bigr) \theta'(t) \diff t
&\ge \int_0^1 \Bigl( (1-\lambda) w\bigl( \sigma(t) \bigr)^{1/(N-1)} +\lambda w\bigl( \tau(t) \bigr)^{1/(N-1)} \Bigr)^{N-1} \\
&\qquad \times \biggl( (1-\lambda) \frac{R(x)}{w(\sigma(t))} +\lambda \frac{R(y)}{w(\tau(t))} \biggr) \diff t \\
&\ge \int_0^1 \bigl( (1-\lambda) R(x)^{1/N} +\lambda R(y)^{1/N} \bigr)^N \diff t \\
&= \bigl( (1-\lambda) R(x)^{1/N} +\lambda R(y)^{1/N} \bigr)^N.
\end{split}
\end{align}
Hence, $R^{1/N}$ is concave as we claimed.

The concavity of $R^{1/N}$ implies
\[ R(x)^{1/N} \le R(0)^{1/N} +(R^{1/N})'(0) x
 \qquad \text{for all}\ x \in (a,b), \]
which can be rewritten as
\begin{equation}\label{eq:R}
R(x) \le R(0) \biggl( 1+\frac{c}{N}x \biggr)^N,
\qquad c:=\frac{R'(0)}{R(0)} =\frac{w(0)}{R(0)}>0.
\end{equation}
In particular, since $R(a)=0$, $a \ge -N/c$ holds.

Now, if $b \ge 1/c$, then we obtain from \eqref{eq:intR} that
\begin{align*}
b &= \int _a^{1/c} R(x)\diff x + \int_{1/c}^b R(x)\diff x\\
&\le \int_{-N/c}^{1/c} R(0)\biggl( 1+\frac{c}{N}x \biggr)^N \diff x + b-\frac{1}{c}\\
&= R(0) \Biggl[ \frac{N}{c(N+1)} \biggl( 1+\frac{c}{N}x \biggr)^{N+1} \Biggr]_{-N/c}^{1/c} + b -\frac{1}{c}\\
&= \frac{R(0)}{c} \biggl(\frac{N+1}{N}\biggr)^N +b-\frac{1}{c},
\end{align*}
where in the inequality we used \eqref{eq:R}, $a \ge -N/c$, and $R \le 1$.
This exactly gives $R(0) \geq (N/(N+1))^N$ which concludes the proof.
In the other case of $b<1/c$, we deduce from \eqref{eq:R} that
\[ 1 \le R(0) \biggl( 1+\frac{c}{N}b \biggr)^N
 < R(0) \biggl( 1+\frac{c}{N}x \biggr)^N \]
for $x \in (b,1/c)$.
Therefore, we similarly observe
\begin{align*}
b &= \int _a^b R(x)\diff x
 \leq \int_{-N/c}^{1/c} R(0)\biggl( 1+\frac{c}{N}x \biggr)^N \diff x - \biggl( \frac{1}{c} -b \biggr)\\
&= \frac{R(0)}{c} \biggl(\frac{N+1}{N}\biggr)^N +b-\frac{1}{c}.
\end{align*}
This completes the proof of the first inequality in \eqref{eq:1d_pos}.
The second one is obtained in the same way or by reversing the interval $(a,b)$.
\end{proof}

We can show that only the cone-like model space \eqref{eq:model} achieves equality in \eqref{eq:1d_pos} (up to translation and dilation).

\begin{lemma}[Rigidity on intervals; $N>1$]\label{lm:=_pos}
If equality holds in the former inequality in \eqref{eq:1d_pos}, then, for some $c>0$, we have $a=-N/c$, $b=1/c$, and
\begin{equation}\label{eq:=_pos}
w(x) =c \biggl( \frac{N}{N+1} \biggr)^N \biggl( 1+\frac{c}{N}x \biggr)^{N-1}.
\end{equation}
Similarly, if equality holds in the latter inequality in \eqref{eq:1d_pos}, then, for some $c>0$, we have $a=-1/c$, $b=N/c$, and
\[
w(x) =c \biggl( \frac{N}{N+1} \biggr)^N \biggl( 1-\frac{c}{N}x \biggr)^{N-1}.
\]
\end{lemma}

\begin{proof}
We deduce from the proof of Lemma~\ref{lm:Grunbaum1d} that $a=-N/c$ and $b=1/c$ necessarily hold.
Moreover, we have equality in \eqref{eq:R} for all $x \in (a,b)$.
Together with $R(1/c)=1$, we obtain
\[ R(x) =\biggl( \frac{N}{N+1} \biggr)^N \biggl( 1+\frac{c}{N}x \biggr)^N, \]
as well as
\[ w(x) =R'(x) =c \biggl( \frac{N}{N+1} \biggr)^N \biggl( 1+\frac{c}{N}x \biggr)^{N-1}. \]
The case of the latter inequality in \eqref{eq:1d_pos} is seen by reversing the interval $(a,b)$.
\end{proof}

One can regard \eqref{eq:=_pos} that $w$ has the cone structure with the apex $a=-N/c$.

\begin{remark}\label{rm:H}
\begin{enumerate}[(a)]
\item 
Note that, in the proof of Lemma~\ref{lm:Grunbaum1d}, the concavity of $R^{1/N}$ (more precisely, the inequality \eqref{eq:R}) is the essential ingredient.
This leads to a slight generalization of one-dimensional Gr\"unbaum's inequality assuming only \eqref{eq:R}.
In higher dimensions, however, we do not know a suitable condition guaranteeing \eqref{eq:R} on almost every needle.

\item
In the proof of Lemma~\ref{lm:=_pos}, testing equality in \eqref{eq:R} was sufficient to obtain \eqref{eq:=_pos}.
In particular, we did not need a characterization of equality in the Borell--Brascamp--Lieb inequality.
\end{enumerate}
\end{remark}

\subsection{Case of $N \in (-\infty,-1) \cup \{\infty\}$}\label{ssc:1d_neg}

Next, we consider $N=\infty$ and $N<-1$, in a similar way to $N>1$.

\begin{lemma}[Gr\"unbaum's inequality on intervals; $N<-1$, $N=\infty$]\label{lm:1d_neg}
Let $-\infty \le a<0<b \le \infty$ and $w\colon (a,b) \lra [0,\infty)$ be a nonnegative integrable function satisfying
$\int_a^b w(x)\diff x=1$ and $\int_a^b xw(x)\diff x=0$.
\begin{enumerate}[{\rm (i)}]
\item\label{1d_neg}
If $w^{1/(N-1)}$ is convex for some $N<-1$, then we have
\begin{equation}\label{eq:1d_neg}
\int_a^0 w(x) \diff x \geq \biggl( \frac{N}{N+1} \biggr)^N, \qquad
\int_0^b w(x) \diff x \geq \biggl( \frac{N}{N+1} \biggr)^N.
\end{equation}
\item\label{1d_inf}
If $\log w$ is concave, then we have
\begin{equation}\label{eq:1d_inf}
\int_a^0 w(x) \diff x \geq \e^{-1}, \qquad
\int_0^b w(x) \diff x \geq \e^{-1}.
\end{equation}
\end{enumerate}
\end{lemma}

Without loss of generality, we will assume that $w>0$ on $(a,b)$.
Recall that, as we mentioned at the beginning of this section, the interval $(a,b)$ can be unbounded.

\begin{proof}
First of all, if $b=\infty$, then we cut off and normalize $w$ as
\[ w_k := \biggl( \int_{a_k}^k w(s) \diff s \biggr)^{-1} \cdot w|_{(a_k,k)} \]
for (large) $k \in \N$, where $a_k \in (a,0)$ is chosen so as to satisfy that the barycenter of $w_k \diff x$ is $0$.
Then, $X_k:=((a_k,k),|\cdot|,w_k \diff x)$ again satisfies the assumptions, and \eqref{eq:1d_neg} or \eqref{eq:1d_inf} for the original space can be obtained as the limit of those for $X_k$ as $k \to \infty$.
Hence, without loss of generality, we will assume $b<\infty$ and, by the same reasoning, $a>-\infty$.

As in the proof of Lemma~\ref{lm:Grunbaum1d}, we set $R(x):=\int_a^x w(s) \diff s$, fix $x,y\in (a,b)$ and $\lambda\in (0,1)$, and define the functions $\sigma, \tau, \theta$.

\eqref{1d_neg}
In place of \eqref{eq:BBL}, we deduce from the convexity of $w^{1/(N-1)}$ and the H\"older inequality of the form
\begin{align}\label{eq:holder1}
\begin{split}
(1-\lambda) \alpha_1 \alpha_2 +\lambda \beta_1 \beta_2
&\le \bigl( (1-\lambda) \alpha_1^{(N-1)/N} +\lambda \beta_1^{(N-1)/N} \bigr)^{N/(N-1)} \\
&\quad \times \bigl( (1-\lambda) \alpha_2^{1-N} +\lambda \beta_2^{1-N} \bigr)^{1/(1-N)}
\end{split}
\end{align}
for $\alpha_1,\alpha_2,\beta_1,\beta_2 \ge 0$ (valid for $N<0$) that
\begin{align*}
R\bigl( (1-\lambda)x +\lambda y \bigr)
&=\int_0^1 w\bigl( \theta(t) \bigr) \theta'(t) \diff t \\
&\ge \int_0^1 \Bigl( (1-\lambda) w\bigl( \sigma(t) \bigr)^{1/(N-1)} +\lambda w\bigl( \tau(t) \bigr)^{1/(N-1)} \Bigr)^{N-1} \\
&\qquad \times \biggl( (1-\lambda) \frac{R(x)}{w(\sigma(t))} +\lambda \frac{R(y)}{w(\tau(t))} \biggr) \diff t \\
&\ge \int_0^1 \bigl( (1-\lambda) R(x)^{1/N} +\lambda R(y)^{1/N} \bigr)^N \diff t \\
&= \bigl( (1-\lambda) R(x)^{1/N} +\lambda R(y)^{1/N} \bigr)^N
\end{align*}
by taking $\alpha_1=R(x)^{1/(N-1)}$, $\alpha_2=(R(x)/w(\sigma(t)))^{1/(1-N)}$ in \eqref{eq:holder1}.
Thus, $R^{1/N}$ is convex, which yields
\begin{equation}\label{eq:H_neg}
R(x)^{1/N} \ge R(0)^{1/N} \biggl( 1+\frac{c}{N}x \biggr), \qquad
c:=\frac{w(0)}{R(0)}>0.
\end{equation}

When $b \ge 1/c$, we obtain from \eqref{eq:intR} that
\begin{align}
b &= \int _a^b R(x)\diff x
 \le  \int_a^{1/c} R(0)\biggl( 1+\frac{c}{N}x \biggr)^N \diff x +b-\frac{1}{c} \label{eq:b_neg}\\
&= R(0) \Biggl[ \frac{N}{c(N+1)} \biggl( 1+\frac{c}{N}x \biggr)^{N+1} \Biggr]_a^{1/c} +b-\frac{1}{c} \nonumber\\
&\le \frac{R(0)}{c} \biggl(\frac{N+1}{N}\biggr)^N +b-\frac{1}{c}, \nonumber
\end{align}
which yields the former inequality in \eqref{eq:1d_neg}.
If $b<1/c$, then we infer from
\[ 
1 \le R(0) \biggl( 1+\frac{c}{N}b \biggr)^N < R(0) \biggl( 1+\frac{c}{N}x \biggr)^N
\]
for $x \in (b,1/c)$ that
\begin{align*}
b &= \int _a^b R(x)\diff x
 \le \int_a^{1/c} R(0)\biggl( 1+\frac{c}{N}x \biggr)^N \diff x - \biggl( \frac{1}{c} -b \biggr)\\
&\le \frac{R(0)}{c} \biggl(\frac{N+1}{N}\biggr)^N +b-\frac{1}{c}.
\end{align*}
This completes the proof of the former inequality of \eqref{eq:1d_neg}.
The latter inequality is seen by reversing the interval $(a,b)$.

\eqref{1d_inf}
Since the concavity of $\log w$ implies the convexity of $w^{1/(N-1)}$ for all $N \in (-\infty,-1)$, we can derive \eqref{eq:1d_inf} from \eqref{eq:1d_neg} by letting $N \to -\infty$.
\end{proof}

We proceed to the equality case, which is more involved than Lemma~\ref{lm:=_pos} due to the unboundedness of the interval.

\begin{lemma}[Rigidity on intervals; $N<-1$, $N=\infty$]\label{lm:=_inf}
\begin{enumerate}[{\rm (i)}]
\item\label{1d=_neg}
If equality holds in the former inequality in \eqref{eq:1d_neg},
then, for some $c>0$, we have $a=-\infty$, $b=1/c$, and
\begin{equation}\label{eq:=_inf}
w(x) =c \biggl( \frac{N}{N+1} \biggr)^N \biggl( 1+\frac{c}{N}x \biggr)^{N-1}.
\end{equation}
Similarly, if equality holds in the latter inequality in \eqref{eq:1d_neg},
then, for some $c>0$, we have $a=-1/c$, $b=\infty$, and
\[
w(x) =c \biggl( \frac{N}{N+1} \biggr)^N \biggl( 1-\frac{c}{N}x \biggr)^{N-1}.
\]

\item\label{1d=_inf}
If equality holds in the former inequality in \eqref{eq:1d_inf},
then we have $a=-\infty$, $b=1/c$, and $w(x)=c \e^{cx-1}$ for some $c>0$.
Similarly, if equality holds in the latter inequality in \eqref{eq:1d_inf},
then we have $a=-1/c$, $b=\infty$, and $w(x)=c \e^{-cx-1}$ for some $c>0$.
\end{enumerate}
\end{lemma}

\begin{proof}
We will consider only the former inequalities in \eqref{eq:1d_neg} and \eqref{eq:1d_inf},
since the latter inequalities can be handled by reversing the interval.

\eqref{1d=_neg}
We first assume $b<\infty$.
In this case, we deduce from the proof of Lemma~\ref{lm:1d_neg} that $a=-\infty$, $b=1/c$, and that equality holds in \eqref{eq:H_neg}.
Combining these with $R(1/c)=1$ implies
\[ R(x) =\biggl( \frac{N}{N+1} \biggr)^N \biggl( 1+\frac{c}{N}x \biggr)^N, \]
and hence
\[ w(x) =R'(x) =c \biggl( \frac{N}{N+1} \biggr)^N \biggl( 1+\frac{c}{N}x \biggr)^{N-1}. \]

Next, we show that equality never holds when $b=\infty$.
Assume in contrary that equality holds in the former inequality in \eqref{eq:1d_neg}.
For (large) $k \in \N$, we take $a_k \in (a,0)$ and $w_k$ as in the proof of Lemma~\ref{lm:1d_neg}, and put
\[ R_k(x) :=\int_{a_k}^x w_k(s) \diff s, \qquad
 c_k:=\frac{R'_k(0)}{R_k(0)} =\frac{w_k(0)}{\int_{a_k}^0 w_k(s) \diff s}. \]
Observe from
\[ \lim_{k \to \infty} c_k = \frac{w(0)}{\int_a^0 w(s) \diff s} =:c \]
that $k \ge 1/c_k$ for large $k$.
Then, in the estimate \eqref{eq:b_neg} in the proof of Lemma~\ref{lm:1d_neg},
\[ k-\frac{1}{c_k} -\int_{1/c_k}^k R_k(x) \diff x =\int_{1/c_k}^k \bigl( 1-R_k(x) \bigr) \diff x \]
necessarily tends to $0$ as $k \to \infty$.
This implies that $\int_a^{1/c} w(s) \diff s =1$,
which is a contradiction since we assumed $w>0$ on $(a,b)$.

\eqref{1d=_inf}
Let us begin with a direct proof of \eqref{eq:1d_inf} under $b<\infty$.
We use the same notations as Lemma~\ref{lm:1d_neg}.
It follows from the concavity of $\log w$ that
\begin{align*}
R\bigl( (1-\lambda)x +\lambda y \bigr)
&= \int_0^1 w\bigl( \theta(t) \bigr) \theta'(t) \diff t \\
&\ge \int_0^1 w\bigl( \sigma(t) \bigr)^{1-\lambda} w\bigl( \tau(t) \bigr)^{\lambda}
 \biggl( (1-\lambda) \frac{R(x)}{w(\sigma(t))} +\lambda \frac{R(y)}{w(\tau(t))} \biggr) \diff t \\
&= \int_0^1 \biggl\{ (1-\lambda) R(x) \biggl( \frac{w(\tau(t))}{w(\sigma(t))} \biggr)^{\lambda}
 +\lambda R(y) \biggl( \frac{w(\sigma(t))}{w(\tau(t))} \biggr)^{1-\lambda} \biggr\} \diff t.
\end{align*}
Together with the concavity of $\log$ and Jensen's inequality, we obtain the concavity of $\log R$:
\begin{align*}
\log R\bigl( (1-\lambda)x +\lambda y \bigr)
&\ge (1-\lambda) \log\biggl[ R(x) \int_0^1 \biggl( \frac{w(\tau(t))}{w(\sigma(t))} \biggr)^{\lambda} \diff t \biggr] \\
&\quad +\lambda \log\biggl[ R(y) \int_0^1 \biggl( \frac{w(\sigma(t))}{w(\tau(t))} \biggr)^{1-\lambda} \diff t \biggr] \\
&\ge (1-\lambda) \biggl( \log R(x) +\int_0^1 \lambda \log\biggl[ \frac{w(\tau(t))}{w(\sigma(t))} \biggr] \diff t \biggr) \\
&\quad +\lambda \biggl( \log R(y) +\int_0^1 (1-\lambda) \log\biggl[ \frac{w(\sigma(t))}{w(\tau(t))} \biggr] \diff t \biggr) \\
&= (1-\lambda) \log R(x) +\lambda \log R(y).
\end{align*}
Therefore,
\[ \log R(x) \le \log R(0) +cx, \qquad c:=\frac{w(0)}{R(0)}>0. \]
When $b \ge 1/c$, we have
\[
b \le \int_a^{1/c} R(0) \e^{cx} \diff x +b-\frac{1}{c}
 \le \frac{R(0)}{c} \e +b-\frac{1}{c},
\]
which yields \eqref{eq:1d_inf}.
In the case of $b<1/c$, since $1 \le R(0) \e^{cb} \le R(0) \e^{cx}$ for $x \ge b$,
we similarly find
\[
b \le \int_a^{1/c} R(0) \e^{cx} \diff x -\biggl( \frac{1}{c}-b \biggr)
 \le \frac{R(0)}{c} \e +b-\frac{1}{c}.
\]

If equality holds, then we have $a=-\infty$, $b=1/c$, and $\log w$ is an affine function.
Put $w(x) =\alpha \e^{\delta x}$ ($\alpha,\delta>0$) and observe that
\[ R(b) =\frac{\alpha}{\delta}\e^{\delta/c} =1, \qquad
 c =\frac{w(0)}{R(0)} =\delta. \]
Therefore, we obtain $\delta=c$, $\alpha=c\e^{-1}$, and $w(x)=c\e^{cx-1}$.
One can also show that equality never holds when $b=\infty$ in a similar way to \eqref{1d=_neg}.
\end{proof}

\begin{remark}\label{rm:L2}
Note that $w$ in \eqref{eq:=_inf} satisfies $w(x) \diff x \in \mathcal{P}^2((-\infty,1/c))$ for $N<-2$, but $w(x) \diff x \in \mathcal{P}^1((-\infty,1/c)) \setminus \mathcal{P}^2((-\infty,1/c))$ for $N \in [-2,-1)$.
\end{remark}

\section{Gr\"{u}nbaum’s inequality for $N \in (1,\infty)$}\label{sc:proof}

We give two proofs of Theorem~\ref{thm:main}, one via localization and one without.
We will need both for our rigidity result (Theorem~\ref{th:rigid}).

Let $(X,\dist_X,\fm)$, $\mu=\rho\fm$, $x_0 \in X$ and $\gamma\colon \R \lra X$ be as in Theorem~\ref{thm:main}.
Recall from Proposition~\ref{pr:proj} that $T(x)=(\bb_\gamma(x),\Pi_Y(x))$, where $\Pi_Y \colon X\lra Y$ is the projection to $Y$, is an isometry from $X$ to $\R \times Y$, namely
\begin{equation}\label{eq:prod}
\dist_X^2(x_1,x_2) =\bigl( \bb_{\gamma}(x_1)-\bb_{\gamma}(x_2) \bigr)^2 +\dist_Y^2\bigl( \Pi_Y(x_1),\Pi_Y(x_2) \bigr)
\end{equation}
for all $x_1,x_2\in X$.
In what follows, via the isometry $T$, we identify $(X,\dist_X,\fm)$ with the product of $(\R,|\cdot|,\diff x)$ and $(Y,\dist_Y,\fn)$ ($\fm$ is identified with $\diff x\otimes \fn$).

By translating $\gamma$, we can assume that $\bb_{\gamma}(x_0)=0$ without loss of generality.
Then, we have the following key observation.

\begin{lemma}\label{lm:0mean}
We have
\[ \int_X \bb_{\gamma} \diff \mu=0. \]
\end{lemma}

\begin{proof}
Let $\eta\colon \R \lra X$ be the geodesic given by $\eta(t):=(t,\Pi_Y(x_0))$.
Note that $\eta(0)=x_0\ (=(\bb_\gamma(x_0),\Pi_Y(x_0)))$ by our choice $\bb_{\gamma}(x_0)=0$.
Then we deduce from \eqref{eq:prod} that, provided that $\mu \in \mathcal{P}^2(X)$,
\begin{align*}
&\frac{\diff}{\diff t} \biggl[ \int_X \dist_X^2 \bigl( \eta(t),x \bigr) \,\mu(\diff x) \biggr]_{t=0} \\
&= \frac{\diff}{\diff t} \biggl[ \int_X \bigl\{ \bigl( t-\bb_{\gamma}(x) \bigr)^2 +\dist_Y^2 \bigl( \Pi_Y(x_0),\Pi_Y(x) \bigr) \bigr\} \,\mu(\diff x) \biggr]_{t=0} \\
&= -2 \int_X \bb_{\gamma}(x) \,\mu(\diff x).
\end{align*}
Since $x_0$ is a barycenter of $\mu$, the left hand side coincides with $0$.
If $\mu$ is only of finite first moment, then we differentiate
\[
\int_X \bigl\{ \dist_X^2\bigl( \eta(t),x \bigr) -\dist_X^2(z_0,x) \bigr\} \,\mu(\diff x)
\]
with an arbitrarily fixed point $z_0 \in X$, and obtain the same conclusion.
This completes the proof.
\end{proof}

\subsection{First proof without localization}\label{ssc:1st}

The first proof of Theorem~\ref{thm:main} is based on Proposition~\ref{pr:proj}, which ensures that $\bb_\gamma\#\mu$ has a density $w$ with respect to the Lebesgue measure on $\R$, and $w^{1/(\beta -1)}$ ($\beta>N$), $\log w$ ($\beta=\infty$), or $-w^{1/(\beta-1)}$ ($\beta<-1$) is concave.
Moreover, by Lemma~\ref{lm:0mean} above, $0$ is its barycenter.
Hence, our one-dimensional analysis in the previous section (Lemmas~\ref{lm:Grunbaum1d}, \ref{lm:1d_neg} with $N=\beta$) yields the result.

\subsection{Second proof via localization}\label{ssc:2nd}

The second proof makes use of the localization (a.k.a.\ needle decomposition), and we additionally assume $\mu \in \mathcal{P}^2(X)$ (unless $\beta>N$).
The localization technique provides a decomposition of a space into a family of geodesics (called \emph{needles}) in such a way that those geodesics inherit some geometric information of the original space.
Through this decomposition, one can reduce a problem into its one-dimensional counterpart.
We refer to \cite[Theorem~5.1]{CM} and \cite{Kl} for the precise statement and more information.

Put $\Omega:=\supp(\mu) =\overline{X \setminus \rho^{-1}(0)}$, which is convex (and bounded when $\beta>N$; recall Remark~\ref{rm:main}).
Thanks to Lemma~\ref{lm:0mean}, we can employ $f:=\bb_{\gamma}\rho$ as a conditional function (as $f$ in \cite[Theorem~5.1]{CM}).
We remark that $f \dist_X(x_0,\cdot) \in L^1(\fm)$ (equivalently, $\bb_\gamma \dist_X(x_0,\cdot) \in L^1(\mu)$)
since $\mu \in \mathcal{P}^2(X)$ and $\bb_{\gamma}$ is $1$-Lipschitz.

We denote the resulting decomposition by $\{(X_q,\fm_q)\}_{q \in Q}$,
where $\fm_q$ is a probability measure on $X_q \subset X$.
The set $X_q$ is the image of a minimal geodesic (so-called a needle) and carries the natural distance structure as the restriction of $\dist_X$.
We also have a measure $\fq$ on $Q$
and $\fq$-almost every needle $(X_q,\fm_q)$ satisfies $\CD(0,N)$ as well as
\begin{equation}\label{eq:1D}
\int_{X_q} \bb_{\gamma} \rho \diff\fm_q =0.
\end{equation}
To be precise, the decomposition is done except a set $Z \subset X$ such that $f=0$ $\fm$-almost everywhere in $Z$.
In the current setting, $Z$ is $\mu$-negligible since $\fm(\bb_{\gamma}^{-1}(0))=0$.
Then, for every $\phi \in L^1(\fm)$ with $\supp(\phi) \subset \Omega$, we have
\[
\int_\Omega \phi \diff\fm
=\int_Q \int_{X_q} \phi \diff\fm_q \,\fq(\diff q).
\]

We deduce from \eqref{eq:1D} that, for $\fq$-almost every $q \in Q$, one of the following holds:
\begin{enumerate}[(1)]
\item $\fm_q(X_q \cap \Omega) =0$;
\item $\fm_q(X_q \cap \Omega) >0$ and $X_q \subset \bb_{\gamma}^{-1}(0)$;
\item $\fm_q(X_q \cap \Omega) >0$, $X_q \not\subset \bb_{\gamma}^{-1}(0)$,
and $X_q \cap \bb_{\gamma}^{-1}(0)$ is a singleton which is the unique barycenter of $\mu_q:=\rho \fm_q$.
\end{enumerate}
We remark that $\mu_q$ is not necessarily a probability measure and that the support of $\mu_q$ is included in $\Omega$.
We infer from
\[
\int_Q \fm_q \bigl( \bb_{\gamma}^{-1}(0) \cap \Omega \bigr) \,\fq(\diff q)
=\fm \bigl( \bb_{\gamma}^{-1}(0) \cap \Omega \bigr)=0
\]
that the case (2) is $\fq$-negligible.
In the case (3), we can identify $X_q$ and $\bb_{\gamma}(X_q) \subset \R$ via $\bb_{\gamma}$.
Then, if we write $\mu_q =w_q \diff x$ through this identification, it follows from Proposition~\ref{pr:proj} (or, more directly, from the H\"older inequality as in Lemma~\ref{lm:beta}; see also the proof of Theorem~\ref{th:rigid} below) that $w_q^{1/(\beta -1)}|_{w_q^{-1}((0,\infty))}$ ($\beta>N$), $\log w_q$ ($\beta=\infty$), or $-w_q^{1/(\beta -1)}$ ($\beta<-1$) is concave.

Therefore, applying Lemmas~\ref{lm:Grunbaum1d}, \ref{lm:1d_neg} (with $N=\beta$) to each normalized needle $(X_q,\mu_q(X_q)^{-1} \mu_q)$ satisfying (3), we obtain
\begin{align}
\mu_q \bigl( \{ x \in X_q \mid \bb_{\gamma}(x) \le 0 \} \bigr)
&\ge \biggl( \frac{\beta}{\beta +1} \biggr)^\beta \cdot \mu_q(X_q), \label{eq:mu_q<}\\
\mu_q \bigl( \{ x \in X_q \mid \bb_{\gamma}(x) \ge 0 \} \bigr)
&\ge \biggl( \frac{\beta}{\beta +1} \biggr)^\beta \cdot \mu_q(X_q) \nonumber
\end{align}
when $\beta>N$ or $\beta<-1$, and
\begin{align*}
\mu_q \bigl( \{ x \in X_q \mid \bb_{\gamma}(x) \le 0 \} \bigr)
&\ge \e^{-1} \cdot \mu_q(X_q), \\
\mu_q \bigl( \{ x \in X_q \mid \bb_{\gamma}(x) \ge 0 \} \bigr)
&\ge \e^{-1} \cdot \mu_q(X_q)
\end{align*}
when $\beta=\infty$.
Integrating these inequalities in $q$ with respect to $\fq$ completes the proof of Theorem~\ref{thm:main}, 
since $\int_Q \mu_q(A) \,\fq(\diff q) =\mu(A)$ for measurable sets $A \subset X$.

\begin{remark}\label{rm:1vs2}
On the one hand, the second proof provides a more detailed control at the level of needles (under $\mu \in \mathcal{P}^2(X)$).
On the other hand, in the first proof, we have a direct connection between $X$ and $\R$ via $\bb_{\gamma}$.
To consider the rigidity problem, we need both viewpoints to integrate the one-dimensional information on needles into a global picture of $X$.
\end{remark}

\subsection{Rigidity}\label{ssc:rigid}

Now, with the help of Lemma~\ref{lm:=_pos}, we study when equality holds in the generalized Gr\"unbaum's inequality \eqref{eq:main} with $\beta>N$.
Put $\Omega :=\supp(\mu)$ as in the previous subsection, and recall that $\bb_\gamma(x_0)=0$.

\begin{theorem}[Rigidity; $N>1$]\label{th:rigid}
In the situation of Theorem~$\ref{thm:main}$, let $(Q,\fq)$, $(X_q,\fm_q)$, and $\mu_q=\rho\fm_q=w_q \diff x$ be the elements of the localization as in Subsection~$\ref{ssc:2nd}$.
Suppose $\beta>N$ and that equality holds in \eqref{eq:main}.
Then, there exists $c>0$ such that, for $\fq$-almost every needle $q \in Q$, we have
$\bb_{\gamma}(X_q \cap \Omega)=[-\beta/c,1/c]$ and
\begin{equation}\label{eq:rigidneedles}
\frac{1}{\mu_q(X_q)} w_q(x)
 = c\biggl( \frac{\beta}{\beta +1} \biggr)^\beta
 \biggl( 1+\frac{c}{\beta} \bb_\gamma(x) \biggr)^{\beta -1}
\end{equation}
for all $x \in X_q \cap \Omega$.
Moreover, regarding $A_t := \bb_{\gamma}^{-1}(t) \cap \Omega$ as a set in $Y$, we have
\begin{equation}\label{eq:rigid_A_t}
\fn(A_t) =\biggl( \frac{\beta +ct}{\beta +1} \biggr)^{N-1}
 \fn(A_{1/c}), \qquad
[\rho\fn](A_t) =\frac{c\beta}{\beta +1}
 \biggl( \frac{\beta +ct}{\beta +1} \biggr)^{\beta -1}
\end{equation}
for all $t \in [-\beta/c,1/c]$.
\end{theorem}

\begin{proof}
We shall analyze by combining both of the two proofs above.
On the one hand, by the first proof in Subsection~\ref{ssc:1st} and Lemma~\ref{lm:=_pos} (with $N=\beta$), we deduce that $\bb_{\gamma}\# \mu=w \diff x$ satisfies $\supp(w)=[-\beta/c,1/c]$ and
\begin{equation*}
w(x) =c \biggl( \frac{\beta}{\beta +1} \biggr)^\beta \biggl( 1+\frac{c}{\beta}x \biggr)^{\beta -1}
\end{equation*}
for some $c>0$.
On the other hand, in the second proof, $\fq$-almost every needle necessarily satisfies equality in \eqref{eq:mu_q<}, thereby, we obtain $\bb_{\gamma}(X_q \cap \Omega)=[-\beta/c_q,1/c_q]$ and
\[
\frac{1}{\mu_q(X_q)} w_q(x) = c_q \biggl( \frac{\beta}{\beta +1} \biggr)^\beta \biggl( 1+\frac{c_q}{\beta}x \biggr)^{\beta -1}
\]
for some $c_q>0$ and all $x \in X_q \cap \Omega$, where we identified $X_q$ and $\bb_{\gamma}(X_q)$ via $\bb_{\gamma}$ as in the previous subsection.
For $\fq$-almost every $q \in Q$, since $[-\beta/c_q,1/c_q] \subset [-\beta/c,1/c]$, we have $c_q \ge c$.
Moreover, by comparing the density at $0$ and $1/c$, we infer that $c_q=c$ necessarily holds.
Therefore, for $\fq$-almost every $q \in Q$,
we obtain $\bb_{\gamma}(X_q \cap \Omega)=[-\beta/c,1/c]$ and \eqref{eq:rigidneedles}.
The latter equation in \eqref{eq:rigid_A_t} is now straightforward by integrating \eqref{eq:rigidneedles} in $q \in Q$:
\begin{align*}
[\rho\fn](A_t)
&=\int_Q w_q(t) \,\fq(\diff q)
=c\biggl( \frac{\beta}{\beta +1} \biggr)^{\beta}
\biggl( 1+\frac{c}{\beta}t \biggr)^{\beta -1} \int_Q \mu_q(X_q) \,\fq(\diff q) \\
&= c\biggl( \frac{\beta}{\beta +1} \biggr)^{\beta}
\biggl( 1+\frac{c}{\beta}t \biggr)^{\beta -1}.
\end{align*}

Next, we have a closer look on the concavity of $w_q^{1/(\beta-1)}$.
Letting $\fm_q =\zeta_q \diff x$, since $(X_q,\fm_q)$ is a $\CD(0,N)$-space, we know that $\zeta_q^{1/(N-1)}$ is concave.
Combining this with the concavity of $\rho^{1/(\beta -N)}$ along $X_q$ and noting $w_q=\rho \zeta_q$, by a similar calculation to Lemma~\ref{lm:beta}, we obtain
\begin{align*}
w_q \bigl( (1-\lambda)x +\lambda y \bigr)^{1/(\beta -1)}
&\ge \Bigl( (1-\lambda)\rho(x)^{1/(\beta -N)}
 +\lambda \rho(y)^{1/(\beta -N)} \Bigr)^{(\beta -N)/(\beta -1)} \\
&\quad \times \Bigl( (1-\lambda)\zeta_q(x)^{1/(N-1)}
 +\lambda \zeta_q(y)^{1/(N-1)} \Bigr)^{(N-1)/(\beta -1)} \\
&\ge (1-\lambda)[\rho\zeta_q](x)^{1/(\beta -1)}
 +\lambda [\rho\zeta_q](y)^{1/(\beta -1)} \\
&= (1-\lambda) w_q(x)^{1/(\beta -1)} +\lambda w_q(y)^{1/(\beta -1)}
\end{align*}
for all $x,y \in X_q \cap \Omega$ and $\lambda \in (0,1)$.
Comparing this with \eqref{eq:rigidneedles},
we find that $\rho^{1/(\beta -N)}$ and $\zeta_q^{1/(N-1)}$ are both affine along $X_q \cap \Omega$ and vanish (only) at $-\beta/c$
(precisely, $\rho^{1/(\beta -N)}$ and $\zeta_q^{1/(N-1)}$ on $[-\beta/c,1/c]$ are the restrictions of such affine functions).
Therefore, $A_t$ satisfies
\begin{equation*}
\begin{split}
\fn(A_t) &= \int_Q \zeta_q(t) \,\fq(\diff q)
 =\biggl( \frac{t+(\beta/c)}{(1/c)+(\beta/c)} \biggr)^{N-1}
 \int_Q \zeta_q(1/c) \,\fq(\diff q) \\
&= \biggl( \frac{ct+\beta}{1+\beta} \biggr)^{N-1} \fn(A_{1/c}),
\end{split}
\end{equation*}
which is the former equation in \eqref{eq:rigid_A_t}.
\end{proof}

\begin{remark}\label{rm:rigid}
It seems plausible that, under the hypothesis in Theorem~\ref{th:rigid}, $N=n$ holds and $\Omega$ is isometric to a convex cone in $\R^n$ with $A_{-\beta/c}$ as the apex and $A_{1/c}$ as the base.
Indeed, in the Euclidean case with the Lebesgue measure, it is known that equality in Gr\"unbaum's inequality is achieved only by cones \cite{Gr00}.
Then, a key step is to show that $A_{-\beta/c}$ is a singleton.
A rigidity result \cite[Theorem~4.2]{BK} for the Brunn--Minkowski inequality seems to play a role, however, it is concerned with $L^2$-optimal transports while the transport along needles is $L^1$-optimal.
\end{remark}


\section{Gr\"{u}nbaum’s inequality for $N \in (-\infty,-1) \cup \{\infty\}$}\label{sc:neg}

In this section, we prove Theorem~\ref{th:neg} and the associated rigidity result (Theorem~\ref{th:rigid_neg}) concerning the cases of $N=\infty$ and $N<-1$, in a similar manner to the case of $N>1$ above.
Since the splitting theorem and localization for $\RCD(0,\infty)$-spaces are yet to be known, we restrict ourselves to weighted Riemannian manifolds.

Let $(M,g,\fm)$ be an $n$-dimensional, complete, weighted Riemannian manifold with $\Ric_{\infty} \ge 0$ or $\Ric_N \ge 0$ for $N<-1$, where $\fm=\e^{-\psi} \,\mathrm{vol}_g$ for some $\psi \in C^{\infty}(M)$.
By the monotonicity \eqref{eq:mono}, $\Ric_{\infty} \ge 0$ is a weaker condition than $\Ric_{N'} \ge 0$ with $N' \ge n$, and $\Ric_N \ge 0$ is even weaker.
In fact, under $\Ric_{\infty} \ge 0$, the boundedness of the diameter as in Lemma~\ref{lm:bdd} does not hold true.
An archetypal example is the Gaussian space $(\R,|\cdot|,\e^{-x^2/2} \diff x)$, which satisfies $\Ric_{\infty} \ge 1$.

We begin with a generalization of Proposition~\ref{pr:proj}.

\begin{proposition} \label{prop:proj_neg}
    Let $(M,g,\fm)$ be an $n$-dimensional complete weighted Riemannian manifold of $\Ric_N\geq 0$ for $N=\infty$ or $N<-1$.
    Assume that $\psi$ is bounded above and there is a straight line $\gamma\colon \R \lra M$.
    Then, there exists an $(n-1)$-dimensional weighted Riemannian manifold $(\Sigma,g_{\Sigma},\fn)$ with $\Ric_{N-1}\geq 0$ $(\Ric_{\infty}\geq 0$ if $N=\infty)$ and an isometry $T\colon (M,g) \lra (\R\times \Sigma,\tilde{g})$ such that
    \begin{itemize}
        \item $\tilde g$ is the product of the Euclidean metric on $\R$ and $g;$
        \item $\Pi_{\R}(T(x))=\bb_{\gamma}(x)$ for all $x\in M;$
        \item $T\#\fm =\diff x \otimes \fn$.
    \end{itemize}
    Moreover, let $\mu=\rho\fm$ be a probability measure on $M$ for some measurable function $\rho\colon M\lra [0,\infty)$ and put $\bb_{\gamma}\#\mu =w \diff x$.
    \begin{enumerate}[{\rm (i)}]
    \item\label{key'_neg}
    If $N<-1$ and $\rho^{1/(\beta -N)}$ is concave on $\rho^{-1}((0,\infty))$ for some $\beta \in (N,-1)$, then $w^{1/(\beta-1)}$ is convex.
    \item\label{key'_inf}
    If $N=\infty$ and $\log\rho$ is concave, then $\log w$ is concave.
    \end{enumerate}
\end{proposition}

\begin{proof}  
The first part follows from \cite[Theorem~1.1]{FLZ} for $N=\infty$ and \cite[Corollary~1.3]{Wy} for $N<-1$ (in fact, it is available for $N<1$).
We identify $M$ with $\R\times \Sigma$, and then $w$ is given by
\[
w(t) =\int_\Sigma \rho(t,y) \,\fn(\diff y).
\]
As in the proof of Proposition~\ref{pr:proj}, we fix $a<b$ and $\lambda \in (0,1)$, and define
\[
h(y) := \rho\bigl( (1-\lambda)a +\lambda b,y \bigr), \qquad
f(y) := \rho(a,y), \qquad
g(y) := \rho(b,y)
\]
for $y \in \Sigma$.
Recall also that, for any minimal geodesic $\eta\colon [0,1] \lra \Sigma$,
\[
\lambda \,\longmapsto\, \bigl( (1-\lambda)a +\lambda b,\eta(\lambda) \bigr) \in \R \times \Sigma =M
\]
is a minimal geodesic.

\eqref{key'_neg}
In this case, we make use of a generalization of the Borell--Brascamp--Lieb inequality in \cite{bacher2010borell,CMS01} to $N<-1$ (the estimate below also works for $N<\beta<1$).
We give an outline for thoroughness.
Set
\[
\nu_0 :=\frac{f}{\int_\Sigma f \diff\fn} \fn, \qquad
\nu_1 :=\frac{g}{\int_\Sigma g \diff\fn} \fn.
\]
Then, there is a unique $L^2$-Wasserstein geodesic $(\nu_{\lambda})_{\lambda \in [0,1]}$ from $\nu_0$ to $\nu_1$, and we can write $\nu_{\lambda}=\zeta_{\lambda} \fn$.
On the one hand, it follows from $\Ric_{N-1} \ge 0$ of $(\Sigma,g_{\Sigma},\fn)$ that
\[
\zeta_{\lambda} \bigl( T_{\lambda}(y) \bigr)^{1/(1-N)}
\le (1-\lambda) \biggl( \frac{f(y)}{\int_\Sigma f \diff\fn} \biggr)^{1/(1-N)}
 +\lambda \biggl( \frac{g(T_1(y))}{\int_\Sigma g \diff\fn} \biggr)^{1/(1-N)}
\]
for $\nu_0$-almost every $y \in \Sigma$
(by \cite[(4.7), (4.9)]{Oneg}), where $T_{\lambda}$ denotes the (unique) optimal transport map from $\nu_0$ to $\nu_{\lambda}$ (thereby $T_\lambda \# \nu_0 =\nu_\lambda$).
On the other hand, the assumed concavity of $\rho^{1/(\beta -N)}$ yields
\[
 h\bigl( T_{\lambda}(y) \bigr)^{1/(\beta-N)} \ge (1-\lambda) f(y)^{1/(\beta-N)} +\lambda g\bigl( T_1(y) \bigr)^{1/(\beta-N)}. 
\]
Thus, we have
\begin{align*}
\int_\Sigma h \diff\fn
&\ge \int_\Sigma \frac{h}{\zeta_\lambda} \diff\nu_\lambda
 = \int_\Sigma \frac{h(T_\lambda)}{\zeta_\lambda (T_\lambda)} \diff\nu_0 \\
&\ge \int_\Sigma \bigl( (1-\lambda) f^{1/(\beta-N)} +\lambda g(T_1)^{1/(\beta-N)} \bigr)^{\beta -N} \\
&\qquad \times \biggl\{ (1-\lambda) \biggl( \frac{f}{\int_\Sigma f \diff\fn} \biggr)^{1/(1-N)}
 +\lambda \biggl( \frac{g(T_1)}{\int_\Sigma g \diff\fn} \biggr)^{1/(1-N)} \biggr\}^{N-1} \diff\nu_0 \\
&\ge \biggl\{ (1-\lambda) \biggl( \int_\Sigma f \diff\fn \biggr)^{1/(\beta -1)} +\lambda \biggl( \int_\Sigma g \diff\fn \biggr)^{1/(\beta -1)} \biggr\}^{\beta -1}
\end{align*}
by integrating the H\"older inequality
\begin{align*}
&\biggl\{ (1-\lambda) \biggl( \frac{f}{\int_\Sigma f \diff\fn} \biggr)^{1/(1-N)}
 +\lambda \biggl( \frac{g(T_1)}{\int_\Sigma g \diff\fn} \biggr)^{1/(1-N)} \biggr\}^{1-N} \\
&\le \bigl( (1-\lambda) f^{1/(\beta-N)} +\lambda g(T_1)^{1/(\beta-N)} \bigr)^{\beta -N} \\
&\qquad \times
\biggl\{ (1-\lambda) \biggl( \int_\Sigma f \diff\fn \biggr)^{1/(\beta -1)} +\lambda \biggl( \int_\Sigma g \diff\fn \biggr)^{1/(\beta -1)} \biggr\}^{1-\beta}.
\end{align*}
Since $\beta -1<0$, this yields the convexity of $w^{1/(\beta -1)}$.

\eqref{key'_inf}
In this case, we can apply the Pr\'ekopa--Leindler inequality from \cite[Theorem~1.4]{CMS06} (with $\lambda=0$).
From the concavity of $\log\rho$, for any minimal geodesic $\eta\colon [0,1]\lra \Sigma$, we obtain
\[
h\bigl( \eta(\lambda) \bigr) 
\ge f\bigl( \eta(0) \bigr)^{1-\lambda} g\bigl( \eta(1) \bigr)^{\lambda}.
\]
Hence, the Pr\'ekopa--Leindler inequality on $(\Sigma,g_{\Sigma},\fn)$ (under $\Ric_{\infty} \ge 0$) implies
\[
\int_\Sigma h \diff\fn
\ge \biggl(\int_\Sigma f \diff\fn \biggr)^{1-\lambda} \biggl(\int_\Sigma g \diff\fn \biggr)^{\lambda},
\]
which shows the concavity of $\log w$.
\end{proof}

\begin{remark}\label{rm:psi}
The assumption $\sup_M \psi<\infty$ is indeed necessary for the splitting.
One can easily find a counter-example by considering the squared distance function $\psi=cd^2(x_0,\cdot)$ for some large $c$ in hyperbolic spaces.
It is also possible to slightly weaken the boundedness into the so-called \emph{$\psi$-completeness} condition thanks to \cite[Corollary~6.7]{Wy}.
\end{remark}

Thus, the one-dimensional analysis in Section~\ref{sc:1d} leads us to generalizations of Gr\"unbaum's inequality, as stated in Theorem~\ref{th:neg}.
Again, we give two proofs of Theorem~\ref{th:neg}, one via localization and one without.

\subsection{Proof without localization}\label{ssc:proof1_neg}

It follows from Proposition~\ref{prop:proj_neg}\eqref{key'_inf} that, in the situation of Theorem~\ref{th:neg}\eqref{main2_inf}, $(\supp(w),w \diff x)$ is a $\CD(0,\infty)$-space.
Then Lemma~\ref{lm:1d_neg}\eqref{1d_inf} yields the claim.
Similarly, Theorem~\ref{th:neg}\eqref{main2_neg} follows from Proposition~\ref{prop:proj_neg}\eqref{key'_neg} and Lemma~\ref{lm:1d_neg}\eqref{1d_neg}.

\subsection{Proof via localization}\label{ssc:proof2_neg}

The localization as described in Subsection~\ref{ssc:2nd} has been known by Klartag \cite[Theorem~1.2]{Kl} in this smooth setting.
Then, with Lemma~\ref{lm:1d_neg}, we can prove Theorem~\ref{th:neg} in the same way as Subsection~\ref{ssc:2nd}, under the additional assumption $\mu \in \mathcal{P}^2(M)$.

\begin{remark}\label{rm:N}
As is natural from the monotonicity \eqref{eq:mono} of the weighted Ricci curvature, we have
\[ \biggl( \frac{N}{N+1} \biggr)^N \le \e^{-1} \le \biggl( \frac{N'}{N'+1} \biggr)^{N'} \]
for $N \in (-\infty,-1)$ and $N' \in (1,\infty)$.
Note also that $\lim_{N \uparrow -1} (N/(N+1))^N =0$, thereby, our generalized Gr\"unbaum's inequality may not have a version for, say $N \in [-1,0]$.
\end{remark}

\subsection{Rigidity}\label{ssc:rigid_neg}

The rigidity result (Theorem~\ref{th:rigid}) can also be generalized to the current setting.
Recall that $\Omega=\supp(\mu)$ and $\bb_\gamma(x_0)=0$.

\begin{theorem}[Rigidity; $N<-2$, $N=\infty$]\label{th:rigid_neg}
In the situation of Theorem~$\ref{th:neg}$,
assume $\mu \in \mathcal{P}^2(M)$ and let $(Q,\fq)$, $(M_q,\fm_q)$, and $\mu_q=\rho\fm_q=w_q \diff x$ be the elements of the localization as in Subsection~$\ref{ssc:2nd}$.
\begin{enumerate}[{\rm (i)}]
\item\label{rigid_inf}
Suppose $N=\infty$ and that equality holds in \eqref{eq:main2_inf}.
Then there exists $c>0$ such that, for $\fq$-almost every needle $q \in Q$, we have
$\bb_{\gamma}(M_q \cap \Omega)=(-\infty,1/c]$ and
\[
\frac{1}{\mu_q(M_q)} w_q(x)
 = c\exp\bigl( c\bb_\gamma(x) -1 \bigr)
\]
for all $x \in M_q \cap \Omega$.
Moreover, regarding $A_t := \bb_{\gamma}^{-1}(t) \cap \Omega$ as a set in $\Sigma$, we have
\[
[\rho\fn](A_t) =c\e^{ct-1}
\]
for all $t \in (-\infty,1/c]$.

\item\label{rigid_neg}
Suppose $N<\beta<-2$ and that equality holds in \eqref{eq:main2_neg}.
Then there exists $c>0$ such that, for $\fq$-almost every needle $q \in Q$, we have
$\bb_{\gamma}(M_q \cap \Omega)=(-\infty,1/c]$ and
\[
\frac{1}{\mu_q(M_q)} w_q(x)
 = c \biggl( \frac{\beta}{\beta +1} \biggr)^\beta \biggl( 1+\frac{c}{\beta}\bb_{\gamma}(x) \biggr)^{\beta -1}
\]
for all $x \in M_q \cap \Omega$.
Moreover, we have
\[
[\rho\fn](A_t)
= \frac{c\beta}{\beta +1} \biggl( \frac{\beta +ct}{\beta +1} \biggr)^{\beta -1}
\]
for all $t \in (-\infty,1/c]$.
\end{enumerate}
\end{theorem}

\begin{proof}
In both cases, by using Lemma~\ref{lm:=_inf}, we can follow the lines of the proof of Theorem~\ref{th:rigid} to show the first assertion on $w_q(x)$, and integrating it in $q \in Q$ implies the second assertion.
We remark that, in \eqref{rigid_neg}, $\beta<-2$ is assumed to ensure that $w_q(x) \diff x$ has finite second moment; recall Remark~\ref{rm:L2}.
\end{proof}

\section{Stability estimates}\label{sc:stab}

This section is devoted to the stability problem for our Gr\"unbaum's inequality.
As a generalization of the rigidity, the stability is concerned with what happens when equality nearly holds.
Similarly to the previous sections, we first analyze the one-dimensional case, and use it to study the general case via the localization.
The localization has played a vital role in some stability results, e.g., \cite{CMM,MO1,MO2} on isoperimetric inequalities, \cite{BF22,CMS23} on the spectral gap (Poincar\'e inequality).
A stability result in the Euclidean setting can be found in \cite{Gr00}.

\subsection{Case of $N \in (1,\infty)$}\label{ssc:stab_pos}

We first consider \eqref{eq:1d_pos} with $N>1$.
Let us begin with the following observation.
For $((a,b),|\cdot|,w(x) \diff x)$ as in Lemma~\ref{lm:Grunbaum1d}, an immediate application of \cite[Theorem~6]{fradelizi}, with $\phi(x)=x^2$ and $n=N$, implies
\[
2w(0) \int_0^{1/(2w(0))} x^2 \diff x
\leq \int_a^b x^2 w(x) \diff x
\leq \frac{N^2}{2(N+1)(N+2)} \frac{1}{w(0)^2},
\]
which can be rewritten as 
\begin{equation}\label{eq:rho0bound}
     \frac{2(N+1)(N+2)}{N^2}\int_a^b x^2 w(x) \diff x
     \leq \frac{1}{w(0)^2}
     \leq 12 \int_a^b x^2 w(x) \diff x.
\end{equation}
Observe that, in the rigidity case \eqref{eq:=_pos}, we have
\[
\frac{1}{w(0)^2} =\frac{1}{c^2} \biggl( \frac{N+1}{N} \biggr)^{2N}, \qquad
\int_{-N/c}^{1/c} x^2 w(x) \diff x =\frac{1}{c^2} \frac{N}{N+2}.
\]

\begin{remark}\label{rm:w0bound}
    The right inequality in \eqref{eq:rho0bound} is sharp for $N=\infty$.
    Indeed, for the uniform distribution $w \equiv 1/(2\sqrt{3})$ on $[-\sqrt{3},\sqrt{3}]$, we have $\int_{-\sqrt{3}}^{\sqrt{3}} x^2 w(x) \diff x=1$ and equality holds in the right inequality. 
\end{remark}

\begin{lemma}\label{lm:1dstab_pos}
Let $((a,b),|\cdot|,w(x) \diff x)$ be as in Lemma~$\ref{lm:Grunbaum1d}$, and
put $c:=w(0)/R(0)$.
If
\begin{equation}\label{eq:1d_stab_e}
R(0) \leq (1+\varepsilon) \biggl( \frac{N}{N+1} \biggr)^N
\end{equation}
holds for some $\varepsilon>0$, then we have
\begin{equation}\label{eq:1d_stab}
\int_\R |R-F| \diff x
\le \frac{4\sqrt{3}N^{N+1}}{(N+1)^N} (1+\varepsilon) \bigl( 1-(1+\varepsilon)^{-1/N} \bigr) \biggl( \int_a^b x^2 w(x) \diff x \biggr)^{1/2},
\end{equation}
where \[
F(x) :=\biggl( \frac{N}{N+1} \biggr)^N \biggl( 1+\frac{c}{N}x \biggr)^N
\]
on $[-N/c,1/c]$, $F(x):=0$ on $(-\infty,-N/c)$ and $F(x):=1$ on $(1/c,\infty)$.
\end{lemma}

Note that $F$ is the cumulative distribution function for the rigidity case \eqref{eq:=_pos}.

\begin{proof}
Consider a probability density
\[ u(x)=cR(0)\biggl( 1+\frac{c}{N}x\biggr)^{N-1} \]
on $(-N/c,\bar{b})$, whose cumulative distribution function is given by
\[ U(x) = R(0)\biggl( 1+\frac{c}{N}x \biggr)^{N}. \]
Note that $R(x) \le U(x)$ by \eqref{eq:R}, which implies $\bar{b} \le b$.
Moreover, we find from $U(\bar{b})=1$ that
\[ \bar{b}=\frac{N}{c} \bigl(R(0)^{-1/N}-1 \bigr), \]
and then \eqref{eq:1d_pos} yields $\bar{b}\leq 1/c$.
One can also compute that the barycenter $\xi \in (-N/c,\bar{b})$ of $u(x) \diff x$ is given by
\[
\xi = R(0) \biggl( 1+\frac{c}{N}\bar{b} \biggr)^N
\biggl( \bar{b}-\frac{1}{c}\frac{N}{N+1} \biggl(1+\frac{c}{N}\bar{b}\biggr) \biggr)
 = \frac{N}{N+1} \biggl( \bar{b}-\frac{1}{c} \biggr)
 \le 0.
\]
Similarly to $\int_a^b R(x) \diff x=b$ in \eqref{eq:intR}, we can calculate
\[
\xi =\int_{-N/c}^{\bar{b}} xU'(x) \diff x
=\bar{b} -\int_{-N/c}^{\bar{b}} U(x) \diff x.
\]
Then, recalling $a \ge -N/c$ from the proof of Lemma~\ref{lm:Grunbaum1d} and $\bar{b} \le b$,
we obtain
\begin{align*}
    \int_\R |R-U| \diff x
    &= \int_{-N/c}^a U \diff x + \int_a^{\bar{b}} (U-R) \diff x + \int_{\bar{b}}^b (1-R) \diff x\\
    &= \int_{-N/c}^{\bar{b}} U \diff x -\int_a^b R \diff x +b-\bar{b}
     = -\xi.
\end{align*}

Next, observe from \eqref{eq:1d_pos} that $U(x) \ge F(x)$, and $\int_{-N/c}^{1/c} F(x) \diff x =1/c$ holds similarly to \eqref{eq:intR}.
Combining these, we obtain
\begin{align*}
   \int_\R |U-F| \diff x
   &= \int_{-N/c}^{\bar{b}} (U-F) \diff x + \int_{\bar{b}}^{1/c} (1-F) \diff x\\
   &= \int_{-N/c}^{\bar{b}} U \diff x - \int_{-N/c}^{1/c} F \diff x + \frac{1}{c}-\bar{b}
   = -\xi.
\end{align*}

Now, if $w$ almost attains the Gr\"unbaum bound in the sense of the hypothesis \eqref{eq:1d_stab_e}, then we have
\[
\bar{b} \geq \frac{1}{c} \bigl((N+1)(1+\varepsilon)^{-1/N}-N\bigr),
\]
and hence
\[
\xi \geq \frac{N}{c} \bigl( (1+\varepsilon)^{-1/N}-1 \bigr).
\]
Thus, we deduce from the right inequality in \eqref{eq:rho0bound} and the hypothesis \eqref{eq:1d_stab_e}
that
\begin{align*}
\int_\R |R-F| \diff x
&\le -2\xi
 \le \frac{2NR(0)}{w(0)} \bigl(1-(1+\varepsilon)^{-1/N} \bigr) \\
&\le 4\sqrt{3} (1+\varepsilon) \frac{N^{N+1}}{(N+1)^N} \bigl(1-(1+\varepsilon)^{-1/N} \bigr)
\biggl( \int_a^b x^2 w(x) \diff x \biggr)^{1/2}.
\end{align*}
This completes the proof.
\end{proof}

In view of Subsection~\ref{ssc:1st}, we obtain the following corollary.

\begin{corollary}\label{cr:stab}
In the situation of Theorem~$\ref{thm:main}$\eqref{main_pos} with $\bb_{\gamma}(x_0)=0$, suppose that
\[
\mu \bigl( \{ x \in X \mid \bb_{\gamma}(x) \le 0 \} \bigr)
\le (1+\ve)\biggl( \frac{\beta}{\beta +1} \biggr)^\beta
\]
holds for some $\ve>0$.
Then $w(x) \diff x=\bb_\gamma \# \mu$ satisfies \eqref{eq:1d_stab} with $\beta$ in place of $N$.
\end{corollary}

We next consider another immediate consequence of Lemma~\ref{lm:1dstab_pos} via the localization.
We use the same notations as Subsection~\ref{ssc:2nd}:
$(Q,\fq)$, $(X_q,\fm_q)$, and $\mu_q=\rho\fm_q=w_q \diff x$ induced from the localization built from $\bb_\gamma \rho$.

\begin{proposition}\label{pr:stab}
In the situation of Theorem~$\ref{thm:main}$\eqref{main_pos} with $\bb_{\gamma}(x_0)=0$, suppose that
\[
\mu \bigl( \{ x \in X \mid \bb_{\gamma}(x) \le 0 \} \bigr)
\le (1+\ve)\biggl( \frac{\beta}{\beta +1} \biggr)^\beta
\]
holds for some $\ve>0$.
Then, for each $\delta>0$, there exists $Q' \subset Q$ such that
\begin{equation}\label{eq:Q'}
\mu\Biggl( \bigcup_{q \in Q'} X_q \Biggr) \ge \frac{\delta}{1+\delta}
\end{equation}
holds and every $q \in Q'$ satisfies
\begin{align}\label{eq:stab}
\begin{split}
&\int_{\R} |R_q-F_q| \diff x \\
&\le \frac{4\sqrt{3}\beta^{\beta+1}}{(\beta+1)^\beta} (1+\ve') \bigl( 1-(1+\ve')^{-1/\beta} \bigr)
\biggl( \frac{1}{\mu_q(X_q)} \int_{X_q} \bb_\gamma^2 \diff\mu_q \biggr)^{1/2},
\end{split}
\end{align}
where we set
\begin{equation}\label{eq:ec}
\ve':=\ve +\delta +\ve\delta, \qquad
c_q :=\frac{w_q(0)}{\int_{-\infty}^0 w_q(x) \diff x},
\end{equation}
and $R_q$ and $F_q$ are the cumulative distribution functions for $\mu_q(X_q)^{-1} \cdot \mu_q$ and
\[
c_q \biggl( \frac{\beta}{\beta +1} \biggr)^\beta \biggl( 1+\frac{c_q}{\beta}x \biggr)^{\beta-1} \quad \text{on}\ \biggl[ -\frac{\beta}{c_q},\frac{1}{c_q} \biggr],
\]
respectively.
\end{proposition}

Recall that letting $\bb_{\gamma}(x_0)=0$ loses no generality, and that $X_q$ is identified with $\bb_\gamma(X_q) \subset \R$ via $\bb_\gamma$.

\begin{proof}
In the second proof of Theorem~\ref{thm:main} in Subsection~\ref{ssc:2nd}, we obtained \eqref{eq:main} by integrating \eqref{eq:mu_q<}.
Then the set $Q' \subset Q$, consisting of $q$ with
\[
\mu_q \bigl( \{ x \in X_q \mid \bb_{\gamma}(x) \le 0 \} \bigr)
\le (1+\delta)(1+\ve)\biggl( \frac{\beta}{\beta +1} \biggr)^\beta \mu_q(X_q),
\]
satisfies \eqref{eq:Q'}.
Indeed, if not, then integrating
\[
\mu_q \bigl( \{ x \in X_q \mid \bb_{\gamma}(x) \le 0 \} \bigr)
> (1+\delta)(1+\ve)\biggl( \frac{\beta}{\beta +1} \biggr)^\beta \mu_q(X_q)
\]
for $q \in Q \setminus Q'$ yields
\begin{align*}
\mu \bigl( \{ x \in X \mid \bb_{\gamma}(x) \le 0 \} \bigr)
&> (1+\delta)(1+\ve)\biggl( \frac{\beta}{\beta +1} \biggr)^\beta \int_{Q \setminus Q'} \mu_q(X_q) \,\fq(\diff q) \\
&> (1+\delta)(1+\ve)\biggl( \frac{\beta}{\beta +1} \biggr)^\beta \biggl( 1-\frac{\delta}{1+\delta} \biggr) \\
&= (1+\ve)\biggl( \frac{\beta}{\beta +1} \biggr)^\beta,
\end{align*}
which contradicts the hypothesis.
Then \eqref{eq:stab} follows from Lemma~\ref{lm:1dstab_pos} with $N=\beta$ and $w=w_q/\mu_q(X_q)$.
\end{proof}

\subsection{Case of $N \in (-\infty,-1) \cup \{\infty\}$}\label{ssc:stab_neg}

We saw in Subsection~\ref{ssc:stab_pos} that a key ingredient to derive a stability estimate is the right inequality in \eqref{eq:rho0bound}, which gives a lower bound of $w(0)$ in term of the second moment for centered distributions.
For $N=\infty$, log-concave probability densities have finite moment of any order (by the fact that they have sub-exponential tails; see, e.g., \cite[Section 2.2]{Ledoux}) and \eqref{eq:rho0bound} is still available.
For $N<-1$, however, one-dimensional $\CD(0,N)$-probability densities may not have finite second moment (recall Remark~\ref{rm:L2}).
Moreover, even if the second moment is finite, \eqref{eq:rho0bound} for negative $N$ seems not known.
Thus, in the following counterpart of Lemma~\ref{lm:1dstab_pos}, we can use \eqref{eq:rho0bound} only when $N=\infty$.

\begin{lemma}\label{lm:1dstab_neg}
Let $((a,b),|\cdot|,w(x) \diff x)$ be as in Lemma~$\ref{lm:1d_neg}$, and put $c:=w(0)/R(0)$.
\begin{enumerate}[{\rm (i)}]
\item\label{stab_neg}
Suppose that $w^{1/(N-1)}$ is convex for some $N<-1$.
If
\begin{equation}\label{eq:1d_stab_neg}
R(0) \leq (1+\varepsilon) \biggl( \frac{N}{N+1}\biggr)^N
\end{equation}
holds for some $\varepsilon>0$, then we have
\begin{equation}\label{eq:1d_stab_ne}
\int_\R |R-F| \diff x
\le \frac{2N}{w(0)} \biggl( \frac{N}{N+1} \biggr)^N (1+\varepsilon) \bigl( 1-(1+\varepsilon)^{-1/N} \bigr) ,
\end{equation}
where
\[
F(x) :=\biggl( \frac{N}{N+1}\biggr)^N\biggl(1+\frac{c}{N}x\biggr)^{N}
\]
on $(-\infty,1/c]$ and $F(x):=1$ on $(1/c,\infty)$, which is the cumulative distribution function for the rigidity case of Lemma~$\ref{lm:=_inf}$\eqref{1d=_neg}.

\item\label{stab_inf}
Suppose that $\log w$ is concave.
If
\begin{equation*}
R(0) \leq (1+\ve) \e^{-1}
\end{equation*}
holds for some $\ve>0$, then we have
\begin{equation}\label{eq:1d_stab_infin}
\int_\R |R-F| \diff x
\le \frac{4\sqrt{3}}{\e} (1+\ve) \log (1+\ve) \biggl( \int_a^b x^2 w(x) \diff x \biggr)^{1/2},
\end{equation}
where $F(x) :=\e^{cx-1}$ on $(-\infty,1/c]$ and $F(x):=1$ on $(1/c,\infty)$, which is the cumulative distribution function for the rigidity case of Lemma~$\ref{lm:=_inf}$\eqref{1d=_inf}.
\end{enumerate}
\end{lemma}

\begin{proof}
\eqref{stab_neg}
The proof goes as in Lemma~\ref{lm:1dstab_pos}.
Consider a probability density
\[ u(x)=cR(0)\biggl(1+\frac{c}{N}x\biggr)^{N-1} \]
on $(-\infty,\bar{b})$, whose cumulative distribution function is given by
\[ U(x) = R(0)\biggl(1+\frac{c}{N}x\biggr)^{N}. \]
We infer from \eqref{eq:H_neg} that $R(x) \le U(x)$, and hence $\bar{b} \le b$.
We also find
\[
\bar{b}=\frac{N}{c} \bigl( R(0)^{-1/N}-1 \bigr)
\]
from $U(\bar{b})=1$,
and \eqref{eq:1d_neg} ensures $\bar{b}\leq 1/c$.
The barycenter $\xi \in (-\infty,\bar{b})$ of $u(x) \diff x$ is given in the same way as Lemma~\ref{lm:1dstab_pos} by
\[
\xi = \frac{N}{N+1}\biggl(\bar{b}-\frac{1}{c} \biggr)\le 0, \qquad
\xi =\int_{-\infty}^{\bar{b}} xU'(x) \diff x
=\bar{b} -\int_{-\infty}^{\bar{b}} U(x) \diff x.
\]
Combining the latter with $\int_a^b R(x) \diff x=b$, we deduce that
\[
    \int_\R |R-U|\diff x
    = \int_{-\infty}^a U \diff x + \int_a^{\bar{b}} (U-R) \diff x + \int_{\bar{b}}^b (1-R) \diff x
    = -\xi.
\]
Next, we observe from \eqref{eq:1d_neg} that $U(x) \ge F(x)$, and $\int_{-\infty}^{1/c} F(x) \diff x =1/c$ similarly to \eqref{eq:intR}.
Hence, we find
\[
   \int_\R |U-F| \diff x
   = \int_{-\infty}^{\bar{b}} (U-F) \diff x + \int_{\bar{b}}^{1/c} (1-F) \diff x
   = -\xi.
\]

If $w$ satisfies \eqref{eq:1d_stab_neg}, then we have
\[
\bar{b} \geq \frac{1}{c} \bigl( (N+1)(1+\ve)^{-1/N} -N \bigr), \qquad
\xi \geq \frac{N}{c} \bigl( (1+\ve)^{-1/N} -1 \bigr).
\]
Therefore, we conclude that
\[
\int_\R |R-F| \diff x
\le -2\xi 
\le \frac{2N}{w(0)} \biggl( \frac{N}{N+1} \biggr)^N (1+\ve) \bigl( 1-(1+\ve)^{-1/N} \bigr). \]

\eqref{stab_inf}
By the monotonicity on the dimensional parameter $N$ (cf.\ \eqref{eq:mono}), $w^{1/(N-1)}$ is convex for all $N<-1$.
Therefore, by letting $N \to -\infty$ in \eqref{eq:1d_stab_ne}, we obtain
\[
\int_\R |R-F| \diff x
\le \frac{2(1+\ve)}{w(0)\e} \log (1+\ve),
\]
from which \eqref{eq:1d_stab_infin} immediately follows with the help of \eqref{eq:rho0bound}.
\end{proof}

Besides a corollary analogous to Corollary~\ref{cr:stab}, one can show the following in the same way as Proposition~\ref{pr:stab}.

\begin{proposition}\label{pr:stab_neg}
\begin{enumerate}[{\rm (i)}]
\item
In the situation of Theorem~$\ref{th:neg}$\eqref{main2_neg} with $\bb_{\gamma}(x_0)=0$, suppose that
\[
\mu \bigl( \{ x \in M \mid \bb_{\gamma}(x) \le 0 \} \bigr)
\le (1+\ve)\biggl( \frac{\beta}{\beta +1} \biggr)^\beta
\]
holds for some $\ve>0$.
Then, for each $\delta>0$, there exists $Q' \subset Q$ such that \eqref{eq:Q'} holds $($with $M_q$ in place of $X_q)$ and every $q \in Q'$ satisfies
\[
\int_{\R} |R_q-F_q| \diff x
\le \frac{2\beta^{\beta+1}}{(\beta+1)^\beta} (1+\ve') \bigl( 1-(1+\ve')^{-1/\beta} \bigr)
\frac{\mu_q(M_q)}{w_q(0)},
\]
where we set $\ve'$ and $c_q$ as in \eqref{eq:ec},
and $R_q$ and $F_q$ are the cumulative distribution functions for $\mu_q(M_q)^{-1} \cdot \mu_q$ and
\[
c_q \biggl( \frac{\beta}{\beta +1} \biggr)^\beta \biggl( 1+\frac{c_q}{\beta}x \biggr)^{\beta-1} \quad \text{on}\ \biggl( -\infty,\frac{1}{c_q} \biggr],
\]
respectively.

\item
In the situation of Theorem~$\ref{th:neg}$\eqref{main2_inf} with $\bb_{\gamma}(x_0)=0$, suppose that
\[
\mu \bigl( \{ x \in M \mid \bb_{\gamma}(x) \le 0 \} \bigr) \le (1+\ve)\e^{-1}
\]
holds for some $\ve>0$.
Then, for each $\delta>0$, there exists $Q' \subset Q$ such that \eqref{eq:Q'} holds $($with $M_q$ in place of $X_q)$ and every $q \in Q'$ satisfies
\[
\int_{\R} |R_q-F_q| \diff x
\le \frac{4\sqrt{3}}{\e} (1+\ve') \log(1+\ve') \biggl( \frac{1}{\mu_q(M_q)} \int_{M_q} \bb_\gamma^2 \diff\mu_q \biggr)^{1/2},
\]
where we set $\ve'$ and $c_q$ as in \eqref{eq:ec},
and $R_q$ and $F_q$ are the cumulative distribution functions for $\mu_q(M_q)^{-1} \cdot \mu_q$ and
$c_q \e^{c_q x-1}$ on $(-\infty,1/c_q]$, respectively.
\end{enumerate}
\end{proposition}

\section{Further problems}\label{sc:outro}

We close the article with some further comments and problems.

\begin{enumerate}[(A)]
\item
The results in Section~\ref{sc:neg} could be generalized to $\RCD(0,N)$-spaces with $N=\infty$ or $N \in (-\infty,-1)$.
We remark that the curvature-dimension condition for $N<0$ was defined in \cite{Oneg} (see also \cite{Oneedle} for the case of $N=0$).
However, both the splitting theorem and localization are not known even for $N=\infty$, thereby we need to generalize them or consider a different method.

\item\label{prob:B}
For $\CD(0,N)$-spaces with $N \in (1,\infty)$, though the localization is known by \cite{CM} under the essential non-branching condition, the isometric splitting does not hold in general.
Indeed, $n$-dimensional normed spaces endowed with the Lebesgue measure satisfy $\CD(0,n)$ but do not isometrically split off the real line.
Furthermore, without the essential non-branching condition, even the topological splitting may fail; we refer to \cite{M} for a counter-example.

\item
Finsler manifolds provide examples of $\CD$-spaces with possibly asymmetric distance structures (see \cite{Oint,Obook}).
In this setting, the localization is known by \cite{CM,Oneedle}.
Moreover, a weak splitting theorem can be found in \cite{Osplit};
for example, there is a one-parameter family of isometric translations (generated by $\nabla \bb_\gamma$) in the Berwald case.
Since this splitting is not isometric, we do not have an exact formula as in \eqref{eq:prod} (consider the case of normed spaces), and it is unclear if Lemma~\ref{lm:0mean} can be generalized.

\item
Our rigidity results (Theorems~\ref{th:rigid}, \ref{th:rigid_neg}) show that equality in generalized Gr\"unbaum's inequalities is attained only when the measure $\mu$ possesses a cone structure.
As we discussed in Remark~\ref{rm:rigid}, we also expect that the set $\supp(\mu)$ is also a cone, as in the Euclidean setting.
To achieve this goal, we would need a more geometric argument, possibly with the help of \cite{BK,de2016volume}.

\item
Once the rigidity as above is established, it is natural to expect a corresponding stability estimate (in a more geometric way than Section~\ref{sc:stab}), bounding the volume of the symmetric difference from a cone in a certain sense, akin to \cite{Gr00}.
\end{enumerate}

\textit{Acknowledgements.}
VEB is supported by the French Agence Nationale de la Recherche, as part of the project
ANR-22-ERCS-0015-01.
SO is supported in part by JSPS Grant-in-Aid for Scientific Research (KAKENHI) 22H04942, 24K00523, 24K21511.

\bibliographystyle{plain}
\bibliography{Biblio}

\end{document}